\documentclass{amsart}
\usepackage{graphicx}
\usepackage{bbm}
\usepackage{xcolor}
\usepackage{tikz-cd}
\usepackage{relsize}
\usepackage{cleveref}

\theoremstyle{plain}
\newtheorem{theorem}{Theorem}
\newtheorem{lemma}[theorem]{Lemma}
\newtheorem{example}{Example}
\newtheorem{corollary}[theorem]{Corollary}

\newtheorem{proposition}[theorem]{Proposition}

\newtheorem{definition}{Definition}
\newtheorem{remark}{Remark}

\newcommand{\qref}[1]{(\ref{#1})}

\newcommand{\abs}[1]{\left\vert #1 \right\vert}

\newcommand{\nin}{\not\in}
\newcommand{\norm}[1]{\left \vert \left \vert #1 \right \vert \right \vert}

\newcommand{\distr} {\overset{\text{(d)}}{\longrightarrow}}
\newcommand{\dist}{\overset{\text{(d)}}{=}}

\newcommand{\Qk}[1]{Q^n_{#1}}
\newcommand{\R}{\mathbb{R}}
\newcommand{\be}{\begin{equation}}
\newcommand{\ee}{\end{equation}}
\newcommand{\ba}{\begin{eqnarray}}
\newcommand{\ea}{\end{eqnarray}}
\newcommand{\nn}{\nonumber}
\newcommand{\mean}{m}
\newcommand{\Ito}{It\^{o}}
\newcommand{\PP}{\mathbb{P}}
\newcommand{\newG}{H^0}
\newcommand{\newF}{G}
\newcommand{\newgen}{\mathcal{A}}
\newcommand{\halfplane}{\mathbb H}

\newcommand{\real}{\mathbb R}

\newcommand{\fnR}{f}

\newcommand{\unif}{\overset{\text{unif.}}\longrightarrow}

\newcommand{\rate}{1}
\newcommand{\thetaone}{\theta_1}
\newcommand{\thetatwo}{\theta_2}
\newcommand{\hull}{L}
\newcommand{\mass}{\alpha}
\newcommand{\brexlocal}{L_{\mathbbm e}}

\newcommand{\coshQ}{\cosh^Q}
\newcommand{\Pn}{p^n}
\newcommand{\E}{\mathbbm E}

\newcommand{\scalar}{3}
\newcommand{\const}{C}

\newcommand{\drdist}{M}
\newcommand{\smallterm}{\varphi}

\newcommand{\crad}[1]{\text{crad} \left(#1 \right)}
\newcommand{\treetime}{T_{\mathcal T}}

\global\long\def\SLE{\operatorname{SLE}}

\begin{document}
\title{Scaling limits of branching Loewner evolutions and the Dyson superprocess}
\author{Vivian Olsiewski Healey
and Govind Menon}
\date{June 10, 2025}
\thanks{This work has been supported by the NSF through grants DMS-1246999 and DMS-1928930 (VOH), DMS-1714187, DMS-2107205 and DMS-2407055 (GM). Partial support to VOH has been provided by an AMS Simons Research Enhancement Grant for PUI Faculty.
Partial support to Menon has also been provided by the Simons Foundation (Award 561041), the Charles Simonyi Foundation and the School of Mathematics at the Institute for Advanced Study.}

\begin{abstract} 
This work introduces a construction of conformal processes that combines the theory of branching processes with chordal Loewner evolution. The main novelty lies in the choice of driving measure for the Loewner evolution: given a finite genealogical tree $\mathcal{T}$, we choose a driving measure for the Loewner evolution that is supported on a system of particles that evolves by Dyson Brownian motion at inverse temperature $\beta \in (0,\infty]$ between birth and death events. 

When $\beta=\infty$, the driving measure degenerates to a system of particles that evolves through Coulombic repulsion between branching events. In this limit, the following graph embedding theorem is established: When $\mathcal{T}$ is equipped with a prescribed set of angles, $\{\theta_v \in (0,\pi/2)\}_{v \in \mathcal{T}}$ the hull of the Loewner evolution is an embedding of $\mathcal{T}$ into the upper half-plane with trivalent edges that meet at angles $(2\theta_v,2\pi-4\theta_v,2\theta_v)$ at the image of each vertex $v$. 

We also study the scaling limit when $\beta\in (0,\infty]$ is fixed and $\mathcal{T}$ is a binary Galton-Watson process that converges to a continuous state branching process. We treat both the unconditioned case (when the Galton-Watson process converges to the Feller diffusion) and the conditioned case (when the Galton-Watson tree converges to the continuum random tree). In each case, we characterize the scaling limit of the driving measure as a superprocess. In the unconditioned case, the scaling limit  is the free probability analogue of the Dawson-Watanabe superprocess that we term the Dyson superprocess.
\end{abstract}

\maketitle

\section{Introduction}
\label{sec:intro}

\subsection{Overview}
The \emph{generalized chordal Loewner equation} (or Loewner-Kufarev equation)
\begin{equation}
\label{generalizedLoewner}
\dot g_t(z)=\int_{\real} \frac{\mu_t(d x)}{g_t(z)-x} , \quad g_0(z)=z,
\end{equation}
originally proposed in a different form by Loewner in 1923~\cite{Loewner} and developed by Kufarev \cite{Kufarev}, gives a bijection between continuously increasing families of compact hulls in the upper half-space $\mathbb H$ 
and certain time-dependent real Borel measures $\{\mu_t\}_{t\geq 0}$ \cite{Bauer, Pommerenke1965, Pommerenke}. This correspondence provides a method to encode growth processes in the upper half-plane by measure-valued processes on the real line. Our goal in this work is to introduce a natural flow of measures $\{\mu_t\}_{t\geq 0}$ that constructs a hull with a given branching structure. 

To this end, we study the hulls generated by equation (\ref{generalizedLoewner}) when the driving measure $\mu_t$ is supported on an interacting particle system. The essential ideas in our construction are as follows:
\begin{enumerate}
\item[(a)] The measure $\mu_t$ is composed of a Dirac mass at the location of each particle. 
\item[(b)] The number of particles evolves according to a Galton-Watson birth-death process, with each particle duplicating at its location (birth) or disappearing (death) after an exponential lifetime.
\item[(c)] The evolution of particles in between branching events is given by Dyson Brownian motion at inverse temperature $\beta \in (0,\infty]$. When $\beta=\infty$ we mean that the Dyson Brownian motion has degenerated to the deterministic evolution of a system of particles whose velocity is given by Coulombic repulsion. 
\end{enumerate}
A precise statement of our construction requires some background on planar trees and Loewner evolution, but the main results may be summarized as follows. We show in Theorem~\ref{maintheorem} below that when $\beta=\infty$ these driving measures generate embeddings of trees in the halfplane. We then consider the scaling limit of the driving measures for all $\beta \in (0,\infty]$ under natural assumptions on the underlying Galton-Watson process, obtaining limiting superprocesses with an explicit description. The gap between these results may be explained as follows. When $\beta=\infty$, the geometry of the hulls may be established by adapting classical methods in Loewner theory. When $\beta$ is finite, the (stochastic) Loewner evolution remains well-defined, but a rigorous study of the regularity of the hulls is at an early stage of development.

The choice of Coulombic repulsion, seen as the vanishing noise limit of Dyson Brownian motion,  is fundamental to our construction. It arises naturally from two distinct points of view. First, from the perspective of classical Loewner theory, Coulombic repulsion provides the natural building block for hulls with branching. Second, as discussed in \S\ref{sec:multipleSLE} below, Dyson Brownian arises naturally in the study of multiple chordal Schramm-Loewner evolution ($\SLE_\kappa$) and the vanishing noise limit is the subject of recent work on $\SLE_{0^+}$ and Loewner energy.

The use of Dyson Brownian motion for the driving measure also provides  explicit scaling limits. When the Galton-Watson process is rescaled to converge to the Feller diffusion, we prove the existence of a limiting superprocess which we refer to as the Dyson superprocess (see Theorem~\ref{LimitingMartingaleProblem}\/ and \S \ref{sec:dyson} below). An analogous superprocess is obtained when the Galton-Watson process converges to the continuum random tree (CRT) by conditioning on total population. While these superprocess are rigorously characterized by a limiting martingale problem, they may be understood informally with explicit stochastic partial differential equations.

Let us now describe our theorems with greater precision.  We then situate our work within the broader context of Schramm-Loewner evolution and previous constructions of conformal maps with branching. 

\subsection{Tree Embedding}
Let us first review background material on marked trees following~\cite{LeGall}.

A \emph{plane tree} is a finite rooted tree $\mathcal T$, for which at each vertex the edges meeting there are endowed with a cyclic order. This condition guarantees that a plane tree is a unicellular planar map, i.e. an embedding of a graph in the sphere (or plane), up to orientation preserving homeomorphism, that has exactly one face. A \emph{marked plane tree} is a finite plane tree $\mathcal T$ and a set of markings $\{h_\nu\in \mathbb R_{\geq0}:\nu\in \mathcal T\}$ such that $h_\emptyset=0$ (where $\emptyset$ denotes the root of $\mathcal T$), and if $\eta$ is an ancestor of $\nu$, then $h_\eta<h_\nu$. A marked plane \emph{forest} is a collection of marked plane trees. Generalizing the notion of the (integer-valued) graph distance, the markings determine a metric $d$ on the tree satisfying
\[
d(\nu, \emptyset)=h_\nu.
\]
The distance between any two vertices is then given by \[d(\nu_1, \nu_2)= \left( h_{\nu_1}-h_\eta \right)+ \left(h_{\nu_2}-h_\eta\right),\] where $\eta$ is the first common ancestor of $\nu_1$ and $\nu_2$.

 Using the distance from the root as the time parameter allows us to encode a continuous-time genealogical birth-death process in a marked tree: each individual corresponds to a vertex $\nu$ with time of death $h_\nu$. The birth time of $\nu$ is $h_{p(\nu)}$, where $p(\nu)$ denotes the parent of $\nu$. The \emph{lifetime} of vertex $\nu$ is the length of the edge from $p(\nu)$ to $\nu$:
\begin{equation}\label{lifetime_of_nu}
l_\nu:=d\left(p(\nu), \nu\right)=h_{\nu}-h_{p(\nu)}.
\end{equation}
The genealogy up to time $t$ is then recorded in the subtree of radius $t$:
\[\mathcal T_{[0,t]} = \{\nu\in \mathcal T: h(p(\nu))\leq t\},
\]
and 
\[
\treetime=\max_{\nu \in \mathcal T} h_\nu
\]
can be understood as the extinction time of the population. 
Finally, we let $\mathcal T_t$ denote the set of elements ``alive" at time $t$:
\begin{equation}\label{eqn:T_t}
\mathcal T_t=\{\nu\in \mathcal T: h(p(\nu))\leq t <h(\nu)\}.
\end{equation}

\begin{definition}
\label{defn:graph-embed}
We say that a set $K\subset \mathbb H$ is a \emph{graph embedding} of $\mathcal T$ if there exists a function $ \mathcal E:\mathcal T \to K$ such that the images of vertices are points, the images of edges are simple curves, the images of edges do not intersect, and  
the curve $\mathcal E ([e_1, e_2])$ has endpoints $\mathcal E(e_1)$ and $\mathcal E(e_2)$ for each edge $[e_1,e_2]$.
\end{definition}

\subsection{Chordal Loewner evolution}
We will construct graph embeddings of marked trees as growing families of Loewner hulls, using the distance from the root as the time parameter. We briefly review  chordal Loewner evolution in order to introduce this construction. The main source for this material is~\cite{Lawler}.

A \textit{compact $\halfplane$-hull} is a bounded subset $K\subset \halfplane$ such that $K=\overline K \cap\halfplane$ and $\halfplane\setminus K$ is simply connected. For brevity, we will refer to such sets simply as ``hulls."
If $(K_t)_{t\geq 0}$ is a continuously increasing family of hulls in the upper half-plane, ordered by inclusion, then there is a unique family of real Borel measures $(\mu_t)_{t\geq0}$ such that the unique conformal mappings $g_t:\halfplane\setminus K_t \to \halfplane$ with hydrodynamic normalization\footnote{By \emph{hydrodynamic normalization}, we mean that
	$g_t(z)=z+\frac{b_t}{z}+O\left( \frac{1}{\abs z ^2}  \right), \; z \to \infty$, where $b_t$ is the half-plane capacity of $K_t$.}
satisfy equation (\ref{generalizedLoewner}). (See \cite{Lawler}, Thm 4.6, and \cite{Bauer}.)

Conversely, given an appropriate family of c\`adl\`ag real Borel measures $\mu_t$, the solution $(g_t)_{t\geq 0}$ to equation (\ref{generalizedLoewner}) is called a \emph{Loewner chain}. (When there is a discontinuity in $\mu_t$, the time derivative in (\ref{generalizedLoewner}) is assumed to denote the right derivative.) For each $t$, if $H_t\subseteq \halfplane$ is the domain of $g_t$, then $g_t$ is the unique conformal mapping $g_t: H_t \to \halfplane$ with hydrodynamic normalization. Furthermore, each set $K_t=\halfplane\setminus H_t$ is a compact $\halfplane$-hull; the nested hulls $(K_t)_{t\geq 0}$ are referred to as the \emph {hulls generated by $\mu_t$}.

When the driving measure is $\mu_t=2\delta_{U(t)}$, for a continuous real function $U$, it is a classical question to ask under what circumstances the hull $K_t$ is a \emph{slit}, i.e. a simple curve $\gamma:[0,T]\to \overline {\halfplane}$ such that $\gamma(0)\in \real$ and $\gamma((0,T])\subset\halfplane$. It is shown in \cite{MRLoewner} and \cite{Lind} that if $U$ is H\"older continuous with exponent $\frac{1}{2}$ and $\norm {U}_\frac{1}{2} <4$, then each $K_t$ is a simple curve. These results were generalized to a disjoint union of $n$ simple curves in \cite{SchleissingerThesis}. Since Brownian motion is only $\alpha$-H\"older continuous for $\alpha<\frac{1}{2}$, the result of \cite{Lind} does not apply directly to $\SLE_{\kappa}$ curves; however, if $\kappa \leq 4$, then $\SLE_{\kappa}$ curves are almost surely simple \cite{Rohde&SchrammSLE}.

In order to embed trees as hulls generated by the Loewner equation, we use the multislit condition of \cite{SchleissingerThesis} to guarantee the simple curve property away from branching times, and separately examine the local behavior of the hulls at branching times. At these branching times, the driving functions exhibit a square-root singularity (in the sense of \eqref{eq:basicsystemsoln}), and we analyze the geometry of the hulls via a blow-up argument (\S\ref{sec:confmap}). 
The delicate analysis of the hulls at branching times is required because geometric properties of Loewner hulls are not necessarily preserved under limits. In fact, one of the most important properties of the (single slit) Loewner equation is that the mappings that produce curves are dense in the space of schlicht mappings. 

Although we restrict ourselves to the chordal Loewner equation in this work, it is important to note that there are also \emph{radial} and \emph{whole plane} versions of the equation. The radial version describes conformal mappings on the unit disc instead of the upper half-plane, so that the driving measure is an evolving measure on the unit circle, and the normalization is chosen at $0$ instead of $\infty$. Although the two settings are closely linked, there are subtle differences that arise from normalizing at an interior point rather than a boundary point.

Finally, we note that while we use the results of \cite{SchleissingerThesis} to check that the simple curve property holds away from branching times (in Proposition \ref{thm:toptree}), the same result could be verified using the more sophisticated framework of Loewner energy (originally described for one curve in \cite{Wang_Loewner_Energy, Wang_Invent}; see \cite{Peltola&Wang} for the configurational multiple-curve version). In particular, the recent work \cite{multiradial_LDP}, which proves a finite-time large deviation principle for multiradial Schramm-Loewner evolution as $\kappa \to 0$, includes the proof that multiradial hulls with finite multiradial Loewner energy are always simple radial multichords. (The proof uses an approach similar to \cite{Friz_Shekhar}.) Indeed, the particle system of primary interest in the present work (characterized by the Coulombic repulsion \eqref{Uevolution}) is the chordal analog of the driving function for multiradial Loewner evolution that generates the zero-energy (``optimal’’) configurations for multiradial $\SLE_{0^+}$.

\subsection{Coulombic repulsion and branching}
The basic building block for tree embedding is the hull $K$ that is the union of two straight slits starting at $0$ -- we call this a \emph{wedge} (see \S\ref{sec:wedge}). A simple, but fundamental, insight in our work is that Loewner theory allows us to obtain an explicit description of the driving measures that generate the wedge as a growth process. Since the wedge has dilational symmetry, it must be generated by a driving measure of the form
\be\label{eq:basicmeasure}
\mu_t=\delta_{\zeta_1 \sqrt t} + \delta_{\zeta_2 \sqrt t}
\ee
for some constants $\zeta_1$ and $\zeta_2$, which depend on the angles of the slits. On the other hand, the initial value problem
\be \label{eq:basicsystem}
\dot {V}_1(t)=\frac{\alpha}{V_1(t)-V_2(t)}, \quad
\dot {V}_2(t)=\frac{\alpha}{V_2(t)-V_1(t)}, \quad V_1(0)=V_2(0)=0,
\ee
is exactly solvable and has solution 
\begin{equation}\label{eq:basicsystemsoln}
V_1(t)=-\alpha\sqrt t, \quad V_2(t)= \alpha\sqrt t.
\end{equation}
Setting $\zeta_1 =-\alpha$ and $\zeta_2=\alpha$ in (\ref{eq:basicmeasure}) generates a wedge that is symmetric about the line $\Re(z)=0$. The constants $\zeta_i$, $i=1,2$ and the dependence of the angle of the wedge on the rate constant $\alpha$ are explained in Proposition~\ref{Vdrivingfunctions} below.

The above calculation shows how Coulombic repulsion arises naturally when a classical conformal mapping is viewed through the lens of Loewner evolution. A different approach is provided by random matrix theory. The Dyson Brownian motion with parameter $\beta>0$ related to equation~\eqref{eq:basicsystem} is the unique weak solution to the It\^{o} equation
\be \label{eq:DysonBM(2)}
\begin{aligned}
dV_1 & = \frac{\alpha}{V_1(t)-V_2(t)} \, dt + \sqrt{\frac{2\alpha}{\beta}} dB_1(t)\\
dV_2 & = \frac{\alpha}{V_2(t)-V_1(t)} \, dt + \sqrt{\frac{2\alpha}{\beta}} dB_2(t)
\end{aligned}
\ee
where $B_1$ and $B_2$ are independent standard Brownian motions. Clearly, equation \eqref{eq:basicsystem} is the vanishing noise (i.e. $\beta\to \infty$) limit of equation \eqref{eq:DysonBM(2)}. The parameter $\alpha$ is absent in the usual convention for Dyson Brownian motion, but we include it in our work since it controls the angle at which the hull branches. The ties between Dyson Brownian motion and multiple $\SLE$ are discussed in greater depth after the statement of Theorem~\ref{maintheorem}.

\subsection{The tree embedding theorem}
The fundamental role of the Coulombic repulsion suggests the next theorem (Theorem \ref{maintheorem}), which we prove in \S\ref{sec:confmap}. 
However, to state the result, we require a bit more terminology. Given a marked tree $\mathcal T$, we say that a time-dependent measure $\mu_t$ is a \emph{$\mathcal T$-indexed atomic measure} if
	\begin{equation}\label{drivingmeas}
	\mu_t=\sum_{\nu\in \mathcal T_t} \delta_{U_\nu(t)},
	\end{equation}
	where for each $\nu \in \mathcal T$, the function $U_\nu: [h_{p(\nu)}, h_\nu)\to \real$ is continuous, and
\begin{equation}	
U_\nu \left(h_{p(\nu)}\right)=\lim_{t\uparrow h_{p(\nu)}} U_{p(\nu)}(t).
	\end{equation}
Furthermore, if $\mass_t$ is a non-negative right-continuous real-valued function taking only finitely many values, we say that the functions $U_\nu$ evolve according to \emph{Coulombic repulsion with strength} $\mass_t$ if
\begin{equation}\label{Uevolution}
    \dot U_\nu (t)=\sum_{\substack{\eta\in \mathcal T_t \\ \eta\neq \nu}}\frac{\mass_t}{U_\nu (t)-U_\eta(t)},
\end{equation}
where the time derivative in \qref{Uevolution} is understood as the right-derivative at $t$-values that correspond to branching times $h_\eta$ or jump times for $\mass_t$. Note that the $U_\nu$ need not start from distinct points.

\begin{theorem}[Tree embedding theorem]\label{maintheorem}
Let $\mathcal T=\{(\nu, h_\nu)\}$ be a binary marked plane tree, with $h_\nu\neq h_\eta$ for all $\nu \neq \eta$. Let $\{\theta_\nu\}_{\nu\in \mathcal T}$ be a collection of angles with $\theta_\nu \in \left(0, \frac{\pi}{2}\right)$. Let $\mu_t$ be the $\mathcal T$-indexed atomic measure evolving by Coulombic repulsion with strength \begin{equation}\label{repulsion_strength}
   \mass_t= \frac{\pi}{\theta_{\nu(t)}}-2,
\end{equation}
where 
\begin{equation}
\nu(t)= \mathrm{argmax} \left\{h_{\nu'}: h_{\nu'} \leq t, \nu' \in \mathcal T \right\}.
\end{equation}
	Then for each $s\in [0, \treetime]$, the hull $K_s$ generated by the Loewner equation (\ref{generalizedLoewner}) with driving measure $\mu_t$ is a graph embedding in $\mathbb H$ of the (unmarked) plane tree $\mathcal T_{[0,s]}$.

	Furthermore, for each $\nu\in \mathcal T$, the embedded edges corresponding to $\nu$ and its children meet at trivalent vertices separated by angles $(2\theta_\nu, 2\pi-4\theta_\nu, 2\theta_\nu)$.
\end{theorem}

\begin{corollary}\label{maincor}
	If $\theta^n$ is distributed as a critical binary Galton-Watson tree with exponential lifetimes of mean $\mean_n$, conditioned to have $n$ edges, then with probability one Theorem \ref{maintheorem} holds for $\mathcal T=\theta^n$.
\end{corollary}

\begin{figure}
	\begin{center}
		\includegraphics[scale=.59]{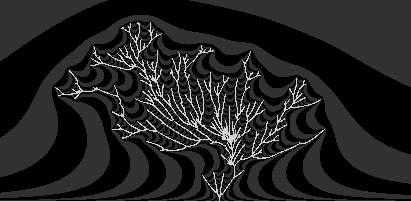}
		\caption{A sample of the random hull generated when $\mathcal T$ is a critical binary Galton-Watson tree with exponential lifetimes and the driving measure is supported on a particle system with Coulombic repulsion, as in Theorem \ref{maintheorem}. (The code that generated this image was written by Brent Werness.)}\label{tree}
	\end{center}
\end{figure}

Theorem \ref{maintheorem} gives a canonical way to embed finite binary plane trees in the upper half-plane with prescribed (symmetric) branching angles at each vertex. Corollary \ref{maincor} specifies that this embedding holds for a specific class of critical binary Galton-Watson trees (an embedded sample of which is shown in Figure \ref{tree}). 

The geometry of the branching angles of the embedded hulls is described precisely in Theorem \ref{thm:angles} (\S \ref{sec:confmap}). Allowing $\mass_t$ to vary gives the freedom to prescribe different branching angles $\theta_\nu$ at different vertices. The case when all $\theta_\nu = \pi/3$ corresponds to a constant repulsion strength of $\mass = \mass_t=1$. In particular, if $\mass=1$, then at each branching time two adjacent curves meet the real line symmetrically at angle sequence $(\frac{\pi}{3}, \frac{\pi}{3},\frac{\pi}{3})$. 
Furthermore, at each branching time $t$, the \emph{left} derivative of $U_\nu(t)$ is bounded, so the Loewner map $g_t$ has a square-root singularity at the driving point $U_\nu(t)$, and the generated curves (before branching) meet the real line perpendicularly (see, for example \cite{KNK, LindMarshallRohde10}).
As a result, the corresponding angles in the embedded tree are doubled: at every vertex of $K_T\in \halfplane$, the three curves meet symmetrically at angles of $2\pi/3$.

In addition to giving a graph embedding of the marked tree $\mathcal T$ as a set in the plane, the Loewner chain of Theorem \ref{maintheorem} preserves the metric information of $\mathcal T$ in the time parameterization: the graph distance from the root in $\mathcal T$ is used as the time coordinate for the Loewner chain. Intuitively, points in $\mathcal T$ whose distance from the root is $s$ are embedded in $\halfplane$ at time $s$.
More precisely, for each $z\in \halfplane$, the solution to (\ref{generalizedLoewner}) exists up to a (possibly infinite) time $T_z$. If $z$ is a point for which $T_z<\infty$, then $z$ corresponds to a point on $\mathcal T$ whose distance from the root is $T_z$.

We do not address the regularity of the embedded curves; in particular, we do not consider whether the curves are quasislits. It appears to be more fruitful to study regularity problems when the Loewner evolution is driven by Dyson Brownian motion for the reasons discussed below.

\subsection{Schramm-Loewner evolution and Dyson Brownian motion}
\label{sec:multipleSLE}
Let us now explain the manner in which our construction builds on links between Schramm-Loewner Evolution ($\SLE_\kappa$) and integrable probability.  As the parameters $\kappa$ (for $\SLE$) and $\beta$ (for Dyson Brownian motion) are related by $\beta=8/\kappa$, we will also briefly highlight the significance of the $\kappa \to 0$ and $\beta \to \infty$ limits in each context.

Originally introduced by Schramm in \cite{Schramm}, $\SLE_\kappa$ has been shown to be the scaling limit of many two-dimensional discrete interfaces that arise in statistical physics, including the loop-erased random walk ($\kappa=2$), the Ising model ($\kappa=3$), flow lines of the Gaussian free field ($\kappa=4$), the percolation exploration process on the triangular lattice ($\kappa =6$), the Peano curve of the uniform spanning tree ($\kappa=8$), and the self-avoiding walk (conjecturally $\kappa=8/3$)  \cite{CN07, CS12, ConfInvLERW, SS09, Smirnov_ICM}.

The connection between Dyson Brownian motion and multiple $\SLE$ was first observed in \cite{CardyCorr, Cardy}. It was further studied in the context of both chordal and radial $\SLE$ in~\cite{delMonaco_chordal, HS_radial, Katori&Koshida}. The behavior of the generated curves depends on the values of the parameters that govern the strength of the Coulombic repulsion $(\alpha)$ and the scaling of the Brownian term $(\kappa)$. 

The first rigorous construction of multiple radial $\SLE_\kappa$ was presented in work by the first author and Lawler~\cite{H-Lawler}. It was shown that for $\kappa\leq 4$, setting $\alpha=\frac{4}{\kappa}$ generates multiple radial $\SLE_\kappa$, while $\alpha=\frac{2}{\kappa}$ generates locally independent $\SLE_\kappa$. (The curves are independent $\SLE$ paths if the interaction term is scaled to $\alpha=0$.) The parameter $\kappa$ determines how rough the $\SLE_\kappa$ paths are, dictating whether the curves are simple, self-intersecting, or space-filling  \cite{Katori&Koshida, Lind,  MRLoewner, Rohde&SchrammSLE}. 

Taking the parameter $\kappa$ to $0$ gives rise to $\SLE_{0+}$. In \cite{Peltola&Wang} it is shown that as $\kappa\to 0$, the $\SLE_\kappa$ curves fluctuate near an $n$-tuple $\boldsymbol \eta$ of ``optimal'' curves that minimize a quantity $I(\boldsymbol \eta)$ called the Loewner energy \cite{Wang_Invent}. Additionally, these optimal curves are the real locus of a real rational function whose poles and critical points flow according to a particular Calogero-Moser integrable system \cite{Alberts_Kang_Makarov} (see also \cite{Zhang_SLE(0)_Calogero-Moser}). Recent work of the first author with Abuzaid and Peltola \cite{multiradial_LDP}  shows that in the multiradial setting the ``optimal'' curves are generated by the radial analog of the Coulombic repulsion considered here in Equation~\eqref{Uevolution}.

Dyson Brownian motion appears as the driving function when multiple $\SLE$ curves are grown simultaneously to a common target point, but we note that many authors have studied multiple $\SLE$ curves grown one at a time and with other link patterns, starting with \cite{Dubedat} and \cite{KL}, and more recently, \cite{BPW_GlobalSLE, JL_partition_function, PWGlobal_Local_SLE}(chordal) and \cite{Z2SLEbdry, Z2SLEint} ($2$-sided radial). Additionally, in settings without branching, other authors have studied the limit of multiple $\SLE$ as the number of curves goes to infinity, showing connections to the complex Burgers equation \cite{delMonaco_chordal, delMonaco&Schleissinger, Hotta&Katori, HS_radial}.

Dyson Brownian motion also serves as a model problem in the second author's program on the Nash embedding theorems~\cite{IM1}.  Dyson's original motivation was the study of the eigenvalues of the self-dual Gaussian ensembles GOU, GUE and GSE respectively. This approach yields Dyson Brownian motion with the parameters $\beta=1$, $2$ and $4$ respectively. The first matrix model for all $\beta \in (0,\infty]$ is a tridiagonal ensemble constructed by Dumitriu and Edelman~\cite{Dumitriu-Edeleman}. Recently, a geometric construction of Dyson Brownian motion within the space of Hermitian matrices for all $\beta \in (0,\infty]$ was presented in~\cite{HIM}. In particular, when $\beta=\infty$, the evolution of particles by Coulombic repulsion corresponds to motion by mean curvature of the associated isospectral orbits. 

The above results suggest rich possibilities for the interaction between multiple $\SLE$ and random matrix theory. However, they also reveal the need to restrict our main theorem on 
tree embedding to $\beta =\infty$ (i.e. $\kappa=0$).  As in the single-slit setting, the sample-paths of Dyson Brownian motion are only $\gamma$-H\"older for $\gamma<1/2$. Thus, they do not satisfy the deterministic condition on the driving function that guarantees simple curves \cite{Lind}. 
Instead, a main tool in the study of the geometric properties of multiple $\SLE$ has been the coupling between multiple $\SLE$ and the Gaussian free field \cite{Dubedat_SLE_GFF, Sheffield_SLE_Quantum_Zipper, ImGeoI}. Using this method, it is shown in \cite{Katori&Koshida} that Loewner evolution driven by non-colliding Dyson Brownian motion starting from distinct points is generated by curves that satisfy the same three phases observed in the single-curve case ~\cite{Rohde&SchrammSLE} (including simple curves for $\kappa\leq 4$). 
This perspective also opens up the possibility of many other connections to random geometry. For example, Miller and Sheffield showed in \cite{ImGeoI} (part of a series with \cite{ImGeoII, ImGeoIII, ImGeoIV}) that a tree structure appears from the interacting flow lines of the Gaussian free field when the flow lines are started from distinct points at the same angle. We hope to address connections between these results and the trees in \Cref{maintheorem} as well as the issue of multiple $\SLE$s starting from the same point in future work.

\subsection{Scaling limits: the Dyson superprocess}
\label{subsec:intro-dyson}
Section \ref{sec:dyson} focuses on finding the scaling limit of the measure-valued processes when the spatial motion (\ref{Uevolution}) is replaced by Dyson Brownian motion. 

In particular, if $\theta$ is a marked plane forest and $\beta\geq 1$, let
\begin{equation}\label{Dysonsystem}
\begin{aligned}
\mu(t)&=\sum_{\nu \in \theta_t} \delta_{U_\nu (t)},
\text{ such that} \\
d U_\nu (t)&= \sqrt{ \frac{2\mass}{\beta}} \,dB_\nu (t) +  \sum_{\substack{\eta\in \theta_t \\ \eta\neq \nu}}\frac{\mass}{U_\nu (t)-U_\eta(t)}\,dt, \text{ and}\\
U_\nu \left(h_{p(\nu)}\right)&=\lim_{t\uparrow h_{p(\nu)}} U_{p(\nu)}(t).
\end{aligned}
\end{equation}

Conditionally on the forest $\theta$, this system has a unique strong solution for any initial configuration such that
$U_\nu(0)\leq U_\eta(0)$ for all $\nu\leq \eta$, with respect to the lexicographical ordering of elements of  $\theta$ (\cite{Anderson_Guionnet_Zeitouni} Proposition 4.3.5, originally \cite{DysonBM}). 

Let $M_F({\mathbb R})$ denote the space of finite measures on $\mathbb R$ endowed with the weak topology, let $D_{M_F(\hat {\mathbb R})}[0,\infty)$ denote the space of c\`adl\`ag paths in $M_F({\mathbb R})$ endowed with the Skorohod topology, and let $\hat{\mathbb R}$ denote the compactification of $\mathbb R$.

We prove the following two theorems, which may be compared to similar results for the Dawson-Watanabe superprocess \cite{Dawson, Konno&Shiga, Reimers}.

For each $n\geq 1$, let $\theta^n$ be distributed as a critical binary Galton-Watson forest with exponential lifetimes of mean $\frac{1}{n}$ starting with $n$ individuals. Let $\{\mu^n\}_{n\geq1}$ be the sequence of measure-valued processes satisfying (\ref{Dysonsystem}) for $\mass_k=\frac{1}{n}$.
Theorem \ref{tightness_unconditioned} shows that the sequence $\{\frac{1}{n}\mu^n\}_{n\geq 1}$ is tight in $D_{\mathcal M_F(\hat {\mathbb R})}[0,\infty)$. The rescaling by $\frac{1}{n}$ is chosen so that the total mass process $\langle 1, \frac{1}{n} \mu^n_t\rangle$ converges in distribution to the Feller diffusion. We characterize the scaling limit in the following theorem.

\begin{theorem}\label{LimitingMartingaleProblem}
	If $\mu_t$ is a subsequential limit of the sequence $\{\mu^k_t\}$ defined in Theorem \ref{tightness_unconditioned}, then $\mu_t \in \mathcal M_F (\mathbb R)$, and for every $\phi \in C^+_b(\mathbb R)\cap \mathcal D(\Delta)$,
	\begin{equation}\label{limit_martingale}
	M_t= \langle \phi, \mu_t \rangle - \langle \phi, \mu_0\rangle - \frac{1}{2} \int_0^t \int_{\mathbb R} \int_{\mathbb R} \frac{\phi'(x) -\phi'(y) }{x-y} \mu_s(dx) \mu_s(dy) ds
	\end{equation}
	is a martingale with quadratic variation
	\begin{equation}
	[M(\phi)]_t=\int_0^t \langle  \phi^2, \mu_s\rangle ds.
	\end{equation}
\end{theorem}

This theorem has a close relationship to the complex Burgers equation. Fix $z \in \mathbb{H}$ and consider the test function $\phi$, Stieltjes transform of $\mu_t$, and covariance kernel defined as follows
\begin{equation}
\label{eq:mg-dyson1}
\phi(x) = \frac{1}{z-x}, \quad f(z,t) = \int_{\mathbb{R}} \frac{1}{z-x} \mu_t(dx),\quad K(z,w;t) = \int_{\mathbb{R}} \frac{1}{z-x}\frac{1}{\bar{w}-x} \mu_t(dx). 
\end{equation}
A direct computation shows that 
\[ M_t = f(z,t) -f(z,0) + \int_0^t (f \partial_z f) (z,s)\, ds\]
is now a martingale taking values in the space of analytic functions on $\mathbb{H}$. In a similar manner, we may fix $w \in \mathbb{H}$, choose $\tilde{\phi}= 1/(w-x)$, denote $\tilde{M}_t$ as the associated martingale, 
and compute the covariation of the martingales $M_t$ and $\tilde{M}_t$ using the polarization identity to obtain
\[ [M,\overline{\tilde{M}}]_t = \int_0^t K(z,w;s)\, ds.\]
Informally, this yields that the evolution of $f$ is given by the stochastic PDE
\begin{equation}
\label{eq:mg-dyson2}
df + f f_z \,dt = dh, \quad dh(z,t) \overline{dh(w,t)} = K(z,w;t) \, dt. 
\end{equation}
where $h_t$ is a centered Gaussian process taking values in the space of analytic functions on $\mathbb{H}$ with covariance listed above.

Another SPDE description is (formally) obtained by recognizing that the complex Burgers equations plays the role of the heat equation in free probability. If we assume that the limiting superprocess $\mu_t$ has density $\rho(x,t)$, we obtain the SPDE 
\begin{equation}\label{eqn:limdensity}
\partial_t \rho + \partial_x \left( \rho \cdot \mathcal H \rho \right)=\sqrt \rho \,\dot W,
\end{equation}
where $\dot W$ is space-time white noise and $\mathcal H$ is the Hilbert transform, defined by
\begin{equation}
\mathcal H \rho(x,t)=\frac{\text{p.v.}}{\pi}\int_{\real} \frac{1}{x-\xi} \rho(\xi,t) d\xi, \quad x\in \real.
\end{equation}
This description is analogous to the SPDE description of the Dawson-Watanabe superprocess (for which the term involving the Hilbert transform is replaced by $-\triangle \rho$). It is known that such a density exists only in one dimension for the Dawson-Watanabe process~\cite{Dawson, Konno&Shiga, Reimers}. At present, we lack a regularity theory for the SPDE~\eqref{eqn:limdensity}. 

\subsection{Conditioned Dyson superprocess and the CRT}
We give an extremely brief introduction to the continuum random tree; the reader is directed to \cite{L&M} or \cite{Pitmanbook} for a comprehensive treatment.

Given a marked plane tree $\mathcal T$, there is an associated excursion $C_{\mathcal T}:[0, 2\sum_{\nu\in \mathcal T} l_\nu]\to \real_{\geq 0}$ called the contour function (or Harris path) of the tree, obtained by tracing the tree at constant speed in lexicographical order beginning at the root and recording whether steps are taken toward or away from the root. The graph distance $d_{gr}$ between two vertices in the tree can be recovered from the contour function as follows. If $v$ and $v'$ are two vertices on the graph, and $s$ and $s'$ are times at which vertices $v$ and $v'$ are visited (according to the contour function construction), then
\begin{equation}\label{graphdist}
d_{gr}(v, v')= C_{\mathcal T}(s)+C_{\mathcal T}(s')-2 \min_{t\in [s, s']} C_{\mathcal T}(t).
\end{equation}

In the same way that deterministic excursions code deterministic marked trees, random trees are coded by random excursions. The \emph{continuum random tree} (CRT) is defined as the random metric tree coded by the normalized Brownian excursion $\mathbbm e:[0,1]\to \real_{\geq 0}$. The CRT can be obtained as a limit of the uniform distribution on finite plane trees as follows [\cite{L&M} Theorem 3.6]. 
	Let $\theta_n$ be uniformly distributed over the set of plane trees with $n$ edges, and equip $\theta_n$ with the graph distance $d_{\text{gr}}$. Then
	\begin{equation}
	\left( \theta_n, \frac{1}{\sqrt{2n}}d_{\text{gr}}\right) \distr \left(\mathcal T_\mathbbm e, d_\mathbbm e\right),
	\end{equation}
	as $n\to \infty$, in the sense of convergence in distribution of random variables with values in the metric space $\mathbb K$ of pointed compact metric spaces, where $\mathbb K$ is equipped with the Gromov-Hausdorff distance.

Theorem \ref{maintheorem} also holds when the $\theta_n$ are distributed as binary Galton-Watson trees with exponential lifetimes of mean $\frac{1}{\sqrt n}$ conditioned to have $n$ edges. (Rescaling the graph distance is not required in this case, as the rescaling is implicit in the decreasing mean lifetimes.) This is the version of the theorem that we will use. The main benefits are (1) it allows us to restrict to binary branching, which is more tractable from the conformal mappings perspective, and (2) in this case the Galton-Watson process is a continuous time Markov process, so we can directly apply the superprocess methods of \cite{Etheridge}.

We then prove the following two analogous results in the conditioned setting, where the trees converge to the continuum random tree. The first result holds for the process stopped when it leaves a localization set, as in \cite{Serlet}.

	 For each $n\geq 1$, let $\theta^n$ be distributed as a critical binary Galton-Watson tree with exponential lifetimes of mean $\frac{1}{\sqrt n}$ conditioned to have $n$ total individuals. 
	 Let $\epsilon\geq 0$, and let $\{ \hat \mu^n \}_{n\geq 1}$ be the sequence of measure-valued processes satisfying (\ref{Dysonsystem}) for $\mass_n=\frac{1}{\sqrt n}$, stopped at time $\sigma$ given by \begin{equation}
	 \sigma=\inf_{t\geq 0}\left\{t: \frac{n-1-\langle 1, \hat \mu^n_t \rangle}{n}\leq \epsilon\right\}.
	 \end{equation}
	The sequence of conditioned processes $\{\frac{1}{\sqrt n}\hat \mu^n\}_{n\geq 1}$ is tight in $D_{\mathcal M_F(\hat {\real})}[0,\infty)$. We show the following conditioned analog of Theorem \ref{LimitingMartingaleProblem}.

\begin{theorem}\label{LimitingMartingaleProblem_Conditioned}
	Let $\{\hat \mu^k\}$ be as above. If $\hat \mu$ is a subsequential limit of $\{\hat \mu^n_{t} \}$,  then $\hat \mu_t \in \mathcal M_F (\real)$, and for every $\phi \in C^+_b(\real)\cap \mathcal D(\Delta)$,
	\begin{equation}
	\hat M_t = \langle \phi, \hat \mu_t\rangle -\langle \phi, \hat \mu_0\rangle - \int_0^t\left( \int_{\real} \int_{\real} \frac{\phi'(x)-\phi'(y)}{x-y} \hat \mu_s(dx) \hat \mu_s(dy) 
	+ \left \langle  \phi\left( \frac{4}{\brexlocal^s} -\frac{\brexlocal^s}{1-\int_0^s \brexlocal^r \, dr}\right), \, \hat \mu_s  \right \rangle \right) ds
	\end{equation}
	is a local martingale with quadratic variation
	\begin{equation}
	[\hat M (\phi)]_t=\int_0^t \langle \phi^2, \hat \mu_s\rangle ds.
	\end{equation}
\end{theorem}
\noindent Here, $\brexlocal^s:=\brexlocal^s(1)$ denotes the local time at level $s$ of the normalized Brownian excursion; see equation \eqref{def:localtime} below for the definition.

We call a solution to this limiting martingale problem a conditioned Dyson superprocess. As in the unconditioned case, we may obtain an SPDE associated to the limiting process by substituting in the test function $\phi(x)=\frac{1}{z-x}$.

\subsection{Conformal processes with branching}
Our work has been greatly stimulated by previous work on conformal maps with branching and some related results in random conformal geometry. Other aggregation models that exhibit branching behavior, including DLA and the Hastings-Levitov model, have been studied using the single slit Loewner equation with discontinuous driving functions (see, for example \cite{Carleson/Makarov} and \cite{Johansson&Sola}, and recently \cite{Berger-Procaccia-Turner}). We have also been motivated by the study of the Brownian map and the closely related problem of finding natural embeddings of the CRT into $S^2$ including \cite{LeGallUniversalityofBMap, MiermontBMap} and \cite{Lin-Rohde}. 
Further fundamental connections between SLE and the CRT have been uncovered in the study of Liouville Quantum Gravity. In particular, the mating-of-trees theorem of Duplantier, Miller, and Sheffield \cite{Duplantier-Miller-Sheffield} shows that Liouville Quantum Gravity surfaces decorated by SLE may be constructed by glueing two copies of the CRT, providing a continuous analog of classical discrete mating-of-trees theorems. 

The main advantages of the construction presented here are as follows.
\begin{enumerate} 
\item[(a)] It provides explicit control on the genealogy of branching. In particular it gives explicit embeddings of finite Galton-Watson trees in the halfplane \emph{as growth processes}. \item[(b)] It fits well with the theory of branching particle systems and superprocesses, yielding the explicit scaling limits described above.
\item[(c)] It incorporates the graph distance from the root as the time parameter: all vertices that are distance $s$ from the root are embedded at time $s$ of the Loewner evolution. Consequently, the metric information of the embedded genealogical tree can be completely recovered from the Loewner chain. 
\end{enumerate}
For these reasons, we believe that the geometric scaling limit of the finite tree embeddings discussed in \ref{sec:model} could provide an embedding of the CRT as a growth process.

\subsection{Acknowledgements} 

VOH would like to thank Greg Lawler and Steffen Rohde for many useful conversations related to the Loewner equation and Brent Werness for generating the simulation shown in Figure \ref{tree}. GM would like to thank J.F. Le Gall for a very stimulating conversation on the embedding problem for the CRT during his visit to Brown University.

\section{Binary branching in conformal maps and Loewner evolution}
\label{sec:binary-branching}

The purpose of this section is to illuminate why Coulombic repulsion appears when branching hulls are generated by Loewner evolution. This section takes a constructive point of view with an emphasis on explicit formulas, so the reader may wish to skip the proofs on a first reading. 

First, we revisit a classical conformal map (the wedge) from the viewpoint of Loewner evolution, explicitly computing its driving function. Then we illustrate the natural interplay between Coulombic repulsion and Loewner evolution with branching by showing that the solution to the system \eqref{Uevolution} locally behaves like the driving function for the wedge. 
Together, these computations explain the origin of Coulombic repulsion and shed light on the role of the repulsive strength $\alpha_t$ from \eqref{Uevolution}, paving the way for the rigorous analysis of the geometry of the Loewner hulls conducted in Sections~\ref{sec:model} and \ref{sec:confmap}.

A central object of our study is Loewner evolution generated by the 
\emph{multislit equation}
\be \label{Loewner(multi)}
\dot{g}_t(z)=\sum_{j=1}^n \frac{1}{g_t(z)-U_j(t)}, \quad g_0(z)=z.
\ee
where $U_i: [0,T] \to \R$, $1 \leq i \leq n$ are measurable, right-continuous functions (not necessarily starting from distinct points). This form of the Loewner evolution corresponds to the driving measure
\begin{equation}
    \label{eq:multislit-driver}
    \mu_t = \sum_{j=1}^n \delta_{U_j(t)}.
\end{equation}
Setting the numerator equal to $1$ in \eqref{Loewner(multi)} is simply a matter of convenience: under the time change $\mathfrak t = \tau t$, equation \eqref{Loewner(multi)} becomes
\[
\partial_{\mathfrak t} g_{\mathfrak t} (z)= \sum_{j=1}^n \frac{\tau}{g_{\mathfrak t}(z)-U_j(\mathfrak t)}, \quad g_0(z)=z.
\] 
We note that in this section, $n$ is fixed, but in Sections~\ref{sec:dyson} and \ref{sec:crt}, the number of particles will be given by a Galton-Watson process $N_t$.

\subsection{Loewner evolution generating a wedge}
\label{sec:wedge}
Let us first demonstrate how a wedge can be generated using~\qref{Loewner(multi)}. Fix two angles 
\begin{equation}
0<\thetaone<\thetatwo<\pi,
\end{equation}
let $\hull_t=\hull_t(\thetaone, \thetatwo)$ (as above) denote the union of two line segments forming angles $\thetaone$ and $\thetatwo$ with the positive real axis from $0$, generated by the multislit Loewner equation (\ref{Loewner(multi)}) up to time $t$ with driving functions $V_1$ and $V_2$. 

An expert will recognize that the dilational symmetry of truncations of $\hull_\infty$ and the well-known Loewner scaling property suggest that there exist constants $\zeta_1$ and $\zeta_2$ (depending on $\thetaone$ and $\thetatwo$) for which the driving functions $\zeta_1\sqrt t$ and $\zeta_2\sqrt t$ generate an increasing family of truncations of $\hull_\infty$ via equation (\ref{Loewner(multi)}), which is the multislit Loewner equation with common parametrization. In Proposition \ref{Vdrivingfunctions} below, we make this idea rigorous, giving explicit formulas for the driving functions.

In order to state Proposition \ref{Vdrivingfunctions}, we will need the following notation. For $0<\thetaone<\thetatwo<\pi$ denote $\thetaone=a\pi$, $\thetatwo=(1-b)\pi$, and
\begin{equation}\label{psi_def}
\begin{aligned}
\psi_1(\thetaone,\thetatwo) &=	\frac{1+x-3a-3bx }{\sqrt{a(1-a)-2abx+b(1-b)x^2}}\\
\psi_2(\thetaone,\thetatwo)&= \frac{\sqrt{(1-a)^2+2x(a+b+ab-1)+x^2(1-b)^2}}{\sqrt{a(1-a)-2abx+b(1-b)x^2}},
\end{aligned}
\end{equation} 
where $x$ is the unique negative root of \begin{equation}\label{origcubic}
Q(x)=-a + a^3 + 3 a x - 3 a^2 x - 3 a b x + 3 a^2 b x + 3 b x^2 - 
3 a b x^2 - 3 b^2 x^2 + 3 a b^2 x^2 - b x^3 + b^3 x^3.
\end{equation}
That $Q(x)$ has a unique negative root follows from computing
\begin{equation}
\begin{aligned}
&Q(x)\underset{x\to-\infty}{\longrightarrow}\infty
\\ &Q(0)<0
\\ &Q(1)>0
\\& Q(x)\underset{x\to\infty}{\longrightarrow}-\infty.
\end{aligned}
\end{equation}

\begin{proposition}\label{Vdrivingfunctions}
	Let $0<\thetaone<\thetatwo<\pi$, and let
	\begin{equation}
	\begin{aligned}
	\zeta_1 &= \psi_1(\thetaone,\thetatwo) - \psi_2(\thetaone,\thetatwo)\\
	\zeta_2 &=  \psi_1(\thetaone,\thetatwo) + \psi_2(\thetaone,\thetatwo),
	\end{aligned}
	\end{equation}
	for $\psi_1(\thetaone,\thetatwo), \psi_2(\thetaone,\thetatwo)$ defined in (\ref{psi_def}) above.
	The hulls generated by the multislit Loewner equation (\ref{Loewner(multi)}) with $N=2$ and driving functions
	\begin{equation}\label{drivingpoints}
	\begin{aligned}
	V_1(t)
	&=\sqrt {t}\, \zeta_1(\thetaone,\thetatwo)\\
	V_2(t)&=\sqrt{ t}\,\zeta_2(\thetaone,\thetatwo),
	\end{aligned}
	\end{equation}
	are the continuously increasing family of truncations of $\hull_\infty(\thetaone,\thetatwo)$, which we denote 
	$(\hull_t)_{t\geq 0}$.
	
\end{proposition}
\begin{proof}
	For $x<0$, the map
	\begin{equation}
	f(z)=(z-1)^a z^{1-a-b} (z-x)^b,
	\end{equation}
	satisfies
	\begin{equation}
	f(x)=f(0)=f(1)=0,
	\end{equation}
	and
	\begin{equation}
	f(\mathbb H)\to \mathbb H\setminus K,
	\end{equation}
	where $K$ is some truncation of $\hull_\infty$. Intuitively, the map $f$ folds the intervals $[x,0]$ and $[0,1]$ into two straight slits in the upper half-plane. Inspired by the folding map $f$, consider the family of maps
	\begin{equation}\label{formoff}
	f_t(z)=\left(z+(a+bx-1)\phi(t) \right)^a \left( z+(a+bx)\phi(t))  \right) ^{1-a-b}\left(z+(a+bx-x)\phi(t)\right)^b,
	\end{equation}
	where $\phi:\mathbb R_{\geq 0} \to \mathbb R_{\geq 0}$ is a differentiable function.
	For each $t$, $f_t$ may be obtained from $f$ by shifting and scaling, so 
	\begin{equation}
	f_t(\mathbb H)=\mathbb H\setminus K_t,
	\end{equation}
	where each $K_t$ is a truncation of $\hull_\infty$. Furthermore, 
	$f_t$ has the hydrodynamic normalization, so we may assume that each $f_t=g_t^{-1}$, where the family of conformal mappings $(g_t)_{t\geq 0}$ satisfies (\ref{Loewner(multi)}) for $N=2$ and continuous real driving functions denoted by $V_1$ and $V_2$. Assuming that the hulls $K_t$ are generated by a Loewner chain guarantees that they are continuously increasing with $t$. In this case, $f_t$ satisfies the \emph{inverse Loewner equation}:
	\begin{equation}\label{invLoewner}
	\dot f_t (z)
	=  f_t'(z) \int_{\real}\frac{\mu_t(d\lambda)}{z-\lambda}, \quad f_0(z)=z.
	\end{equation}
	 which simplifies to
	\begin{equation}\label{particularinvLoewner}
	\frac {\dot f_t (z)}{f_t'(z)}=  \frac{ \rate}{z-V_1(t)} + \frac{ \rate}{z-V_2(t)} , \quad f_0(z)=z.
	\end{equation}
	
	A direct (though lengthy) computation shows that (\ref{formoff}) and (\ref{particularinvLoewner}) together imply that $V_1$ and $V_2$ satisfy (\ref{drivingpoints}). (The $\sqrt {t}$ appears when solving for $\phi(t)$.)
	
	Finally, we note that the cubic $Q$ given by (\ref{origcubic}) has three distinct real roots (since its discriminant is strictly positive for $a,b>0$, $a+b<1$), and the same computations may be carried out if $x$ in (\ref{psi_def}) is chosen to be any of the three roots. However, each root corresponds to a distinct permutation of the angles, and the negative real root is the one that corresponds to the counterclockwise order $(a\pi, (1-a-b)\pi, b\pi)$ that is required in the definition of $\hull_\infty(\thetaone,\thetatwo)$.
\end{proof}

\begin{example}\label{balancedcase}[Balanced case]
	If $0<\thetaone=\pi-\thetatwo<\frac{\pi}{2}$, then $x=-1$, and the hulls $\hull_t$ are generated by (\ref{Loewner(multi)}) with driving functions
	\begin{equation}\label{eqn:balancedcase}
	V_j(t)=(-1)^j \sqrt{t}\sqrt{\frac{\pi-2\thetaone}{\thetaone}}, \quad j=1,2.
	\end{equation}
\end{example}

\subsection{The Embedding Generated by Coulombic Repulsion}\label{subsec:specificembedding}
Given the importance of the driving function $\sqrt \kappa B_t$ in the study of single slit Loewner equation (this is the driving function for which the evolution is $\SLE_\kappa$), it is natural to consider the effect of using Dyson Brownian motion as the driving measure for the multi-slit equation. To conform with the notation of random matrix theory, we continue to use the notation $\kappa = \frac{2\alpha}{\beta}$. Dyson Brownian motion is described by the stochastic differential equation
\begin{equation}\label{dyson}
dx_k= \sum_{j: j\neq k} \frac{\mass}{x_k-x_j}\, dt+\sqrt {\frac{2\alpha}{\beta}} \,dB^k_t,
\end{equation}
where for each $k\in \{ 1, \ldots, N \}$, $B^k$ is an independent linear Brownian motion. Intuitively, if this diffusion is used to prescribe the evolution of the atoms of the driving measure in between branching times, the result should be a kind of $\SLE$-tree.
However, the geometry of multiple $\SLE$ with a common origin point is not yet fully understood, and we do not attempt it here.

Instead, we use only the deterministic part of Equation (\ref{dyson}) in our construction:
\begin{equation}\label{repulsion}
\frac{d x_k}{dt}=\sum_{j:j\neq k}\frac{\mass}{x_k-x_j}.
\end{equation}

Consider the initial value problem
\begin{equation} \label{IVPX}
\begin{aligned}
\dot X(t)&=\left(\sum_{j: j\neq 1}\frac{\mass}{x_1(t)-x_j(t)}, \ldots, \sum_{j:j\neq N}\frac{\mass}{x_N(t)-x_j(t)}  \right)
\\ X(t_0)&=\left(x_1^0, \ldots, x_N^0\right),
\end{aligned}
\end{equation}
where $x_j$ are the coordinates of $X=\left(x_1, \ldots, x_N \right)$,
and $x_1^0\leq x_2^0 \leq \cdots \leq x_N^0$.
That this system has a unique solution even when some $X(t_0)\in \partial \mathbb R^N_>$ is a special case of the existence and uniqueness of a strong solution for Dyson Brownian motion (originally treated in \cite{DysonBM}, see \cite{Anderson_Guionnet_Zeitouni}, Proposition 4.3.5). In fact, equation~\eqref{IVPX} can be interpreted as a gradient descent of entropy~\cite[Thm.2]{HIM}.

We will consider the solution when the initial condition is 
\begin{equation}
x^0_1< \cdots <x^0_k = x^0_{k+1}< \cdots x^0_N.
\end{equation}
In this case, we would like to determine the rate at which the coordinates $x_k$ and $x_{k+1}$ separate from their initial position $x_0:=x_k(t_0)=x_{k+1}(t_0)$. 
We will show that near time $t_0$ the points $x_k(t)$ and $x_{k+1}(t)$ are well-approximated by the driving function for a wedge: $\pm\sqrt \alpha \sqrt{t-t_0}$.
In Section \ref{sec:confmap}, the rate of separation will be used to determine the angles of approach of the Loewner hulls generated by $\mu_t$ defined above in (\ref{drivingmeas}).
\begin{proposition}\label{approach=1}
	Assume that there is a unique index $k$ such that
	\begin{equation}\label{collide}
	x_{k+1}(t_0)=x_k(t_0)
	\end{equation}
	where $X(t)=(x_1(t), \ldots, x_N(t))$ is the unique solution to (\ref{IVPX}).
	Then
	\begin{equation}\label{propxapproach}
	\begin{aligned}
	\lim_{t\downarrow t_0}\frac{x_k(t)-x_k(t_0)}{\sqrt{t-t_0}}&=-\sqrt{\mass}
	\\\lim_{t\downarrow t_0}\frac{x_{k+1}(t)-x_{k+1}(t_0)}{\sqrt{t-t_0}}&=\sqrt{\mass}.
	\end{aligned}
	\end{equation}
\end{proposition}

\begin{proof}
To simplify the notation, denote $x_0=x_{k}(t_0)=x_{k+1}(t_0)$, and let
    \[
    z_1(t)=\frac{x_{k}(t)-x_0}{\sqrt{t-t_0}}, \qquad \text{and} \qquad
    z_2(t)=\frac{x_{k+1}(t)-x_0}{\sqrt{t-t_0}}.
    \]
    We will show that
    \begin{equation}\label{eqn:LRlimits}
    \lim_{t\downarrow t_0} z_1(t)= -\lim_{t\downarrow t_0} z_2(t),
    \end{equation}
    and
\begin{equation}\label{eqn:claim_limit}
\lim_{t\downarrow t_0}z_2(t)=\alpha \lim_{t\downarrow t_0} \frac{1}{z_2(t)}.
\end{equation}
Equation \eqref{eqn:claim_limit} implies that
\[
\lim_{t\downarrow t_0} z_2(t) = \pm \sqrt \alpha.
\]    
Using \eqref{eqn:LRlimits} and the assumption that $x_{k+1}(t) \geq x_k(t)$, we conclude that
\[
\lim_{t\downarrow t_0} z_1(t) = -\sqrt \alpha, \qquad \text \qquad
\lim_{t\downarrow t_0} z_2(t) = \sqrt \alpha.
\]
It remains to verify \eqref{eqn:LRlimits} and \eqref{eqn:claim_limit}.
The key to verifying \eqref{eqn:claim_limit} is that $x_{k+1}$ and $x_k$ are continuously differentiable on $(t_0, t)$, and we have an explicit formula for the derivative. Writing this out, we see that all terms involving the other $x_j$ (for $j\neq k, k+1$) vanish in the limit. To be precise: we may apply l'H\^opital's rule (equivalently, apply Cauchy's mean value theorem, and take the limit) to show that
\[
\begin{aligned}
\lim_{t\downarrow t_0} z_2(t) &= \lim_{t \downarrow t_0} 2\sqrt{t-t_0}\, \dot x_{k+1}(t)\\
& = \lim_{t\downarrow t_0} \frac{2 \alpha \sqrt{t-t_0}}{x_{k+1}(t)-x_k(t)}+ \lim_{t\downarrow t_0}\sum_{j:j\neq k, k+1} \frac{2\alpha \sqrt{t-t_0} }{x_{k+1}(t)-x_j(t)}.
\end{aligned}
\]
Since $x_j(t)$ is bounded away from $x_0$ for all $j\neq k, k+1$, each term in the righthand sum vanishes in the limit as $t \to t_0$, so 
\begin{equation}\label{eqn:z2limit}
\lim_{t\downarrow t_0} z_2(t) = \lim_{t\downarrow t_0} \frac{2 \alpha \sqrt{t-t_0}}{x_{k+1}(t)-x_k(t)}.
\end{equation}
Performing the analogous computation for $z_1(t)$, we see that  
\[
\begin{aligned}
\lim_{t\downarrow t_0} z_1(t) &= \lim_{t\downarrow t_0} \frac{2 \alpha \sqrt{t-t_0}}{x_{k}(t)-x_{k+1}(t)}\\
&=-\lim_{t\downarrow t_0} z_2(t),
\end{aligned}
\]
verifying \eqref{eqn:LRlimits}.
Finally, substituting
\[
x_{k+1}(t)-x_k(t)=\sqrt{t-t_0}\, \big(z_2(t)-z_1(t) \big)
\]
into the denominator of \eqref{eqn:z2limit} and using \eqref{eqn:LRlimits} verifies \eqref{eqn:claim_limit},
completing the proof.
\end{proof}

The next section (\S \ref{sec:model}) generalizes well-known properties of single-slit Loewner evolution to the multislit case, as these properties will be needed to prove Theorem \ref{maintheorem}. 
The analysis of the specific angles at which the tree edges meet will be carried out in \S \ref{sec:confmap}, which culminates in the proof of Theorem \ref{maintheorem}.
These branching angles are given precisely in Theorem \ref{thm:angles}, which relies on Proposition \ref{approach=1} above.

\section{Geometric criteria for multislit evolution}
\label{sec:model}
	In preparation for the proof of Theorem \ref{maintheorem}, we require a few technical lemmas about the geometry of hulls generated by the multislit Loewner equation; these lemmas are the topic of this section. 
	
\subsection{Right continuity of slits at $t=0$}\label{conditionforbranching}
	The first observation is that continuous driving functions on $[0,T]$ that generate simple curves on $[\epsilon, T]$ for all $0<\epsilon<T$ also generate simple curves on $[0,T]$.
	Extending the Loewner equation backwards is fundamentally different from extending it forwards, and no result of this sort holds for forward extension from $[0, T-\epsilon]$ to $[0, T]$. This fundamental difference is a result of the semi-group property of Loewner chains, which implies that $g_\epsilon (K_T)$ is homeomorphic to $K_T\setminus K_\epsilon$. In particular, we have the following two Lemmas concerning the topology of $K_T$ in this setting.
	
	\begin{lemma}\label{simple}
		Let $U:[0,T]\to \mathbb R$ be a continuous function. Let $(g_t)_{t\in[0,T]}$ be the corresponding Loewner chain given by the single-slit Loewner equation 
		\begin{equation}
		    \dot g_t(z)=\frac{2}{g_t(z)-U(t)}, \quad g_0(z)=z,
		\end{equation}
		and let $(K_t)_{t\in[0,T]}$ be the corresponding family of hulls. If $g_s(K_T\setminus K_s)$ is a slit for all $s\in (0,T]$, then $\overline K_T = K_T \cup U(0)$ is a slit.
	\end{lemma}
	
	\begin{proof}
	Without loss of generality, we may assume that $U(0)=0$.
		To show that $K_T$ is generated by a curve, we must verify that for each $t\in [0,T]$,
		\begin{equation}\label{eqn:curvelimit}
		\gamma(t):=\lim_{y\downarrow 0} g^{-1}_t\left(U(t)+iy\right)
		\end{equation}
		exists, and the function $t\mapsto \gamma(t)$ is continuous. 
		For each $t>0$, expressing $g_t=g_{t/2, t} \circ g_{t/2}$, we have
		\begin{equation}\label{curvetip}
		\lim_{y\downarrow 0} g^{-1}_t\left(U(t)+iy\right) = \lim_{y\downarrow 0} \left[g^{-1}_{t/2} \circ g^{-1}_{t/2, t}\left(U(t)+iy\right) \right].
		\end{equation}
		Since $g_{t/2}(K_T\setminus K_{t/2})$ is a simple curve, 
		\begin{equation}
		\lim_{y\downarrow 0} g^{-1}_{\frac{t}{2},t}\left(U(t)+iy\right)
		\end{equation}
		exists and lies in $\mathbb H$, which is the domain of $g^{-1}_{t/2}$. Since $g^{-1}_{t/2}$ is continuous on its domain, we conclude that 
		\[
		\gamma(t):=\lim_{y\downarrow 0} g^{-1}_t\left(U(t)+iy\right) = g^{-1}_{t/2} \left( \lim_{y\downarrow 0} g^{-1}_{\frac{t}{2},t}\left(U(t)+iy\right) \right)
		\]
		exists.		In a similar way, continuity of $\gamma$ on $(0,T]$ follows from the assumption that $g_s(K_T\setminus K_s)$ is a slit for all $s\in (0,T]$. By [\cite{Lawler}, Lemma 4.13] (which we generalize in Lemma \ref{lemma:Lawler_better} below), the diameter of $K_\delta$ is decreasing to $0$ as $\delta\downarrow 0$, so setting $\gamma(0):=0$, the function $t\to \gamma(t)$ is right continuous at $t=0$.

It remains to show that the curve $\gamma$ is simple.
		If not, then there exists $s\in (0,T)$ such that 
		\[g_s(\gamma(s,T])\cap \mathbb R\neq \emptyset,\]
		which violates the assumption that $g_s(K_T\setminus K_s)$ is simple for all $s\in(0,T]$.
	\end{proof}

In order to show that an analog of Lemma \ref{simple} holds for the multislit case, we will need the following lemma, which is a straightforward extension of [\cite{Lawler}, Lemma 4.13]. The importance of the result is the conclusion that for small $t$, each point in the hull $K_t$ is contained in a disk around one of the driving points $U_j(0)$, where the radius of the disk is proportional to $\sqrt t$ (or the largest value for $0\leq s \leq t$ of $\abs{U_j(0)-U_j(s)}$, whichever is larger). This allows us to prove right continuity of the curves at $0$.

\begin{lemma}[Local Growth Property]\label{lemma:Lawler_better}Let $U_1, \ldots, U_n$ be real functions that are continuous on $[0,T]$.
	Let $\fnR_t$ denote the Loewner chain with driving measure
	\[
	\mu_t=\sum_{j=1}^n \delta_{U_j(t)}.
	\]
	Then
for each $z\in K^\fnR_t$, there is at least one index $j\in \{1, \ldots, n\}$ such that $\abs{z-U_j(0)}\leq \scalar \const_{t}$, where
	\begin{equation}
\const_{t}=\const_{t}(\mu)=  \sup\left\{ \abs{U_j(s)-U_j(0)}: 0\leq s \leq t, j\in \left\{1, \ldots, n\right\} \right\} \, \lor \, 
 \sqrt {nt}.
\end{equation}
\end{lemma}

\begin{proof}
	First we will show that if $\abs{z-U_j(0)}>\scalar \const_{t}$, for all $j=1, \ldots n$, then $\abs{\fnR_s(z)-z}\leq C_{t}$, for $0\leq s\leq t$. 
	For each $\abs{z}>\scalar C_{t}$, define the stopping time
	\[
	\sigma=\sigma (z,t)=\min \{ s: \abs{\fnR_s(z)-z}\geq C_{t} \}.
	\]
	Expanding $\abs{z-U_j(0)}$ using the triangle inequality, we see that  if $s<t\land \sigma$, then 
	\begin{equation}
	\abs{\fnR_s(z)-U_j(s)}\geq \abs{z-U_j(0)} -\abs{U_j(s)-U_j(0)} -\abs{\fnR_s(z)-z}\geq \const_t,
	\end{equation}
	so that
	\[
	\abs{\dot \fnR_s(z)}=\abs{\sum_{j=1}^n \frac{1}{\fnR_s(z)-U_j(s)}}\leq \sum_{j=1}^n \frac{1}{\abs{\fnR_s(z)-U_j(s)}}
	\leq \frac{n}{C_{t}},
	\]
	
	which implies that $ \abs{\fnR_s(z)-z} \leq \frac{ ns}{C_{t}}$. By the definition of $\sigma$, this implies that either $\sigma>t$ or $\const_{t}^2\leq n\sigma$. Since $\const_{t}\geq \sqrt {nt}$, we conclude that $\sigma \geq t$. 
	
	Now, for each $z\in K^\fnR_t$, there is $\hat s \in [0,T]$ such that either $\fnR_{\hat s} (z)=U_j(\hat s)$ for some index $j$, or $z$ is ``swallowed'' at time $\hat s$. In the first case, 
\begin{equation}
\abs {z-U_j(0)}\leq \abs{z-U_j(\hat s)} + \abs {U_j(\hat s)-U_j(0)}.
\end{equation}

If $\abs{z-U_j(0)}> \scalar \const$ for every $j=1, \ldots, n$, then the righthand side is $\leq 2\const_t$, which is a contradiction. Therefore, $\abs{z-U_j(0)}\leq \scalar \const$ for at least one index $j$.
\end{proof}
This allows us to extend to the case of multiple curves.

\begin{lemma}\label{lemma:2curves}
	Let $U_1, \ldots, U_n$ be continuous real functions on $[0,T]$. Let $(g_t)_{t\in [0,T]}$ be the Loewner chain for driving measure
	\[
	\mu_t=\sum_{j=1}^n\delta_{U_j(t)},
	\]
	and let $(K_t)_{t\in [0,T]}$ be the corresponding family of Loewner hulls. 
Assume that for every $s \in(0,T]$ the hull $g_s (K_T\setminus K_s)$ is a union of $n$ disjoint slits.
	Then $K_T$ is also a union of $n$ slits that are disjoint in $\mathbb H$. 
\end{lemma}

Note that the conclusion of the lemma is that the slits are disjoint in $\halfplane$, but they are not necessarily disjoint in $\overline \halfplane$. (See Remark \ref{rmk:nonintersecting} below.)

\begin{proof}

As in Lemma \ref{simple}, for $j=1, \ldots, n$, the curves
\[
\gamma_j(t):=\lim_{y\downarrow 0} g_t^{-1}(U_j(t) + iy)
\]
are well-defined for $t\in (0, T]$. The fact that each $\gamma_j$ is simple follows from the same argument as in Lemma \ref{simple}. Furthermore, to see that the curves are disjoint in $\mathbb H$, assume that there is a non-zero point of intersection of $\gamma^k$ with $\gamma^l$. Then there exist times $0<s,t\leq T$ such that
\[
\gamma_k(s)=\gamma_l(t),
\]
But for any $\sigma<s\land t$, the hull $g_\sigma (K_T\setminus K_\sigma)$ is a union of disjoint simple curves, so this situation is impossible.

Finally, right continuity of the $\gamma^j$ at $t=0$ follows from Lemma \ref{lemma:Lawler_better}.

\end{proof}

\begin{remark}\label{rmk:nonintersecting} In the setting of Lemma \ref{lemma:2curves}, it is possible for two or more of the driving functions to start at the same point, i.e. satisfy $U_j(0)=U_k(0)$; this corresponds to the curves starting at the same point on the real line. This will be the case in the next section (\S \ref{sec:confmap}) when we apply Lemma \ref{lemma:2curves} to time intervals that start at a branching time $s_i$. On the other hand, the driving functions cannot intersect at any non-zero time, since this would violate the assumption that $g_s\left(K_T\setminus K_s\right)$ is a union of $n$ disjoint slits.
\end{remark}

\section{Blow-up of binary branching and the proof of \Cref{maintheorem}}
\label{sec:confmap}

The main result of this section is  \Cref{maintheorem}; it says roughly that using Coulombic repulsion for the spatial interaction between particles in the (branching) driving measure generates tree embeddings with prescribed branching angles. 
We begin with \Cref{thm:toptree}, which is the purely topological result that the curves generated by Coulombic repulsion on finite time intervals may be pieced together to form trees. \Cref{thm:angles} explicitly describes the branching angles -- this geometric analysis forms the bulk of the section. Finally, we combine \Cref{thm:toptree} and \Cref{thm:angles} to prove \Cref{maintheorem}.

\begin{proposition}\label{thm:toptree}
Let $\mathcal T=\{(\nu, h_\nu) \}$ be a binary marked plane tree, with $h_\nu \neq h_\eta$ for all $\nu \neq \eta$. Let $\alpha_t: \left[0, \treetime\right] \to \mathbb R$ be a right-continuous function taking only finitely many values. 
Let $\mu_t$ be the $\mathcal T$-indexed atomic measure with Coulombic repulsion with strength $\alpha_t$. Then for each $s\in [0, \treetime]$, the hull $K_s$ generated by the Loewner equation (\ref{generalizedLoewner}) with driving measure $\mu_t$ is a graph embedding in $\halfplane$ of the (unmarked) plane tree $\mathcal T_{[0,s]}$.
\end{proposition}

The statement of \Cref{thm:angles} requires one more definition.  

\begin{definition}\label{def:angles}
Let $(K_t)_{t \in [0,T]}$ be a Loewner hull. We say that $K_t$ meets the real line at angle sequence $\Big(\theta_1, \,\pi-(\theta_1 +\theta_2) ,\,\theta_2 \Big)$ if there exist $x \in \real$ and $R>0$ such that $K_t \cap \mathcal B(x,R)$ has exactly two connected components (denoted by $K^1_{t}$ and $K^2_{t}$), and the following holds. For any $\epsilon>0$ there exists $t_\epsilon$ such that
\begin{equation}
\begin{aligned}
\theta_1-\epsilon &< \arg (z-x) < \theta_1 +\epsilon, \quad &\forall z \in K^1_{t_\epsilon}\\
\pi- (\theta_1+\theta_2)-\epsilon &< \arg (z-x) < \pi-(\theta_1+\theta_2) +\epsilon, \quad &\forall z \in K^2_{t_\epsilon}
\end{aligned}
\end{equation}
\end{definition}

\begin{theorem}\label{thm:angles}
	Let $U_1,\ldots, U_n$ be real functions that are continuous on $[0,1]$ and satisfy
	\begin{equation}
	U_1(t)<\cdots<U_n(t),
	\end{equation}
	for all $t\in [0,1]$, except for the initial condition of $2$ consecutive driving points:
	\begin{equation}
	U_k(0)=U_{k+1}(0)=0,
	\end{equation}
	and assume that $K_1$ is a union of simple curves.
	Furthermore, let $U^\rho_1, \ldots, U^\rho_n$ denote the functions:
	\[
	U^\rho_j(t)=\rho U_j(t/\rho^2).
	\]
	If there exists $\alpha>0$ such that
	\begin{equation}\label{uniformconvergence}
	(U^\rho_{k}(t), \,U^\rho_{k+1}(t)) \unif 
	(-\sqrt{\alpha t}, \,\sqrt{\alpha t})
	\quad\text{on } [0,1] \text{ as }\rho \to \infty,
	\end{equation}
	then the connected components of the hull $K_1$ corresponding to $U_k$ and $U_{k+1}$ are simple curves meeting the real line at angle sequence $(\theta, \pi-2\theta, \theta)$, where
	\[
	\theta=\frac{\pi}{\alpha+2}.
	\]
\end{theorem}

The blow-up assumption in equation (\ref{uniformconvergence}) provides the information on angles necessary to complete the proof of \Cref{maintheorem}.

\subsection{Proof of \Cref{thm:toptree}}
\begin{proof}[Proof of \Cref{thm:toptree}]
	
 Let $\delta>0$ be fixed. First, we will show that the measure $\mu_t$ defined by \eqref{drivingmeas} and \eqref{Uevolution} generates simple curves on the interval $[s_i+\delta, s_{i+1})$. 
It is apparent from \eqref{Uevolution} that $U_\nu$ is continuously differentiable on any interval that does not contain a branching time. This property may be extended to the closed interval $[s_i+\delta, s_{i+1}]$ by using the left-derivative at $s_{i+1}$, implying that $\dot U_\nu$ is bounded on $[s_i+\delta, s_{i+1})$. (Blow-up only occurs for the right-derivative at branching times.) It follows that $U_\nu$ is $\frac{1}{2}$-H\"older continuous on the interval $[s_i+\delta, s_{i+1}]$ with vanishing H\"older norm. By Theorem 3.8.24 in \cite{SchleissingerThesis}, the hull generated on $[s_i+\delta, s_{i+1})$ consists of $n$ connected components, and each is a quasislit. In particular, the evolution generates disjoint simple curves.

By \Cref{lemma:2curves}, we may extend this system backwards in time to the degenerate initial condition,
	so on each interval $[s_i, s_{i+1})$ the generated hull has $\abs{\mathcal T_{s_{i}}}$ connected components, each of which is a simple curve, and two of which share a boundary point.

	 Piecing together the solutions, we conclude that at each $0\leq s\leq T$, the hull $K_s$ is a graph embedding of the subtree 
	\begin{equation}
	\mathcal T_{[0,s]}=\{\nu\in \mathcal T: h_{p(\nu)}<s\}.
	\end{equation}
\end{proof}

\begin{proof}[Proof of \Cref{maincor}]
	The only hypothesis that needs to be checked is that $h_\nu\neq h_\eta$ for all $\nu \neq \eta \in \theta^n$, but this holds with probability one.
\end{proof}

\subsection{Proof of \Cref{thm:angles}}
The proof of \Cref{thm:angles} is based on a comparison between the exact solutions of Section~\ref{sec:binary-branching} and multislit Loewner evolution with driving functions that satisfy equation~\eqref{uniformconvergence}.

\begin{proof}[Proof of Theorem~\ref{thm:angles}]
The proof breaks into three distinct arguments: 
\begin{enumerate}
    \item Hausdorff convergence and Brownian scaling is sufficient to establish the limiting angle sequence.
    \item Gr\"{o}nwall estimates to bound $|g_t^{-1}(z)- (g_t^\rho)^{-1}(z)|$ with explicit dependence on $\Im(z)$. 
    \item The use of Koebe distortion to establish Hausdorff convergence.
\end{enumerate}

{\em Step 1.\/} We first define Hausdorff convergence in the form we need in the claim below. We then explain the manner in which Theorem~\ref{thm:angles}] follows from the claim. This is followed by the technical steps (2) and (3) above.

We begin by considering the behavior of the rescaled driving functions.  For all $\rho$, the $k^{th}$ and $(k+1)^{th}$ driving points satisfy $U^\rho_k(0)=U^\rho_{k+1}(0)=0$. On the other hand, all other driving points get farther away as $\rho$ increases: $\abs{U_j(t)}\to \infty$ as $\rho\to \infty$ for all $t\in[0,1]$ and $j\neq k, k+1$.

As in (\ref{eqn:balancedcase}), let 
\begin{equation}
    V_1(t)=-\sqrt {\alpha t}, \qquad V_2(t)=\sqrt {\alpha t},
\end{equation}
and let $\eta_1, \eta_2$ denote the two parameterized curves (line segments) that comprise the hull generated by driving measure $\delta_{V_1(t)}+ \delta_{V_2(t)}$.
 By construction, the hull $\eta_1\left((0,1]\right) \cup \eta_2\left((0,1]\right)$ meets the real line at angle sequence $(\theta, \pi-2\theta, \theta)$, where
\[
\theta=\frac{\pi}{\alpha+2}.
\]
On the other hand, let $\gamma^\rho_k$ and $\gamma^\rho_{k+1}$ denote the curves corresponding to $U^\rho_k$ and $U^\rho_{k+1}$ in $K^\rho_1$. 

We will prove that $\gamma^\rho_k\left([0,1]\right)$ and $\gamma^\rho_{k+1}\left([0,1]\right)$ converge 
to $\eta_1\left([0,1]\right)$ and $\eta_2\left([0,1]\right)$ in the Hausdorff metric. Notice that this is weaker than showing convergence of the curves pointwise in t.

\begin{minipage}[c]{6in}
\underline{\textbf{Claim:}} 
For each 
$\epsilon>0$, there exists 
$\rho_\epsilon>0$ such that for all $\rho>\rho_\epsilon$,
\begin{equation}\label{eqn:claim}
d_H\left(\gamma^\rho_k \cup \gamma^\rho_{k+1} \cup \R, \eta_1 \cup \eta_2 \cup \R\right)<\epsilon. 
\end{equation}
\end{minipage}

Before verifying the claim, we will explain why the conclusion of the theorem follows from it: since $\eta_1,\eta_2$ meet the real line at angle sequence $\big( \theta, \pi-2\theta, \theta\big)$ by construction, combining the claim with the scaling rule for Loewner hulls finishes the proof. More precisely, we recall the general scaling rule for Loewner hulls, sometimes referred to as ``Brownian scaling'':
if driving measure $\mu_t(\cdot)$ generates hulls $K_t$,  then driving measure  $\mu_{t / \rho^2}(\cdot/\rho)$ generates hulls  $\rho K_{t/\rho^2}$. In particular, for $\hat \rho>\rho$, then $\gamma^{\hat\rho}$ is a scaled and truncated version of $\gamma^\rho$:
\[
\gamma^{\hat \rho}_1=\mathsmaller{\frac{\hat \rho}{\rho}}\cdot \gamma^\rho_{\rho^2/\hat\rho^2}.
\]
Let $r$ denote the length of the segments $\eta_1, \eta_2$:
\[
r=\abs{\eta_1}= \abs{\eta_2},
\]
and let $\epsilon_r$ denote the small angle
\[
\epsilon_r=\sin^{-1}\left(\frac{\epsilon}{2r}\right).
\]
If $\rho$ is large enough that 
claim (\ref{eqn:claim}) holds, then the scaling rule implies that each $z\in \gamma^\rho_k\cup \gamma^\rho_{k+1}$ is contained in one of four possible sectors:
\begin{equation}\label{eqn:sectors}
\arg(z) \in (0, \epsilon_r) \cup \left(\theta-\epsilon_r, \theta+\epsilon_r\right) \cup 
\left(\pi-2\theta-\epsilon_r, \pi-2\theta+\epsilon_r\right)
\cup
(\pi-\epsilon_r, \pi).
\end{equation}
But since $\gamma^\rho_{k+1}$ is simple, it does not revisit $0$ after time $t=0$, so it is entirely contained in just one of these sectors: if $z\in \gamma^\rho_{k+1}$, then
\[\arg(z)\in \left(\theta-\epsilon_r, \theta+\epsilon_r\right).\] Similarly, $\gamma^\rho_k$ is contained in the sector $\left(\pi-2\theta-\epsilon_r, \pi-2\theta+\epsilon_r\right)$. This shows that $\gamma^\rho_k \cup \gamma^{\rho}_{k+1}$ meets the real line at angle sequence $(\theta, \pi-2\theta, \pi)$. 

This concludes the first step in our argument.

{\em Step 2. \/} We now turn to the estimates that establish the claim. The first of these is a Gr\"{o}nwall argument that provides a uniform comparison between $g_t^{-1}(z)$ and $(g_t^\rho)^{-1}(z)$ when $z$ lies in a compact subset of the upper half-plane. Specifically, we assume that $|z|\leq R$ where $R$ is sufficiently large and $\Im z \geq \hat{y}$. Here $\hat{y}>0$ is held fixed in the proof, but may be arbitrarily small. Our final estimate states that for each $\epsilon>0$ there exists $\rho_{\hat y, \epsilon}$ such that
\begin{equation}\label{eqn:estimate-a}
\abs{ g_t^{-1}(z)  - (g_t^\rho)^{-1}_t(z) } < \epsilon, \qquad \forall \rho>\rho_{\hat y, \epsilon}, \Im(z)\geq \hat y.
\end{equation}

In order to establish this estimate, we will use the {\em reverse\/} Loewner evolution as is common in Loewner theory. The advantages are that the reverse Loewner equation is defined on the upper half-plane and it extends continuously to the boundary. The Gr\"{o}nwall argument reduces to a comparison between two distinct reverse Loewner equations. This is (almost) a standard argument in ordinary differential equations. However, we must account for the fact that the Lipschitz constant of the vector field diverges as $\Im(z) \to 0$. This too is common in Loewner theory and we adopt a stopping-time argument that is similar to~\cite[Prop. 4.47]{Lawler} to complete the proof.
 
Let $t\in [0, 1]$ be fixed. Define the functions $h$ and $h^\rho$ to be the conformal mappings that satisfy the \emph{reverse} Loewner equation for $s\in[0,t]$:
\begin{equation}
\label{eq:reverse-l1}
\dot h_s (z)=\frac{-1}{h_s(z)-V_1(t-s)} + \frac{-1}{h_s(z)-V_2(t-s)}, \qquad h_0(z)=z,
\end{equation}
and 
\begin{equation}
\label{eq:reverse-l2}    
\dot h^\rho_s (z)=-\sum_{j=1}^n \frac{1}{h^\rho_s(z)-U^\rho_j(t-s)}, \qquad h^\rho_0(z)=z.
\end{equation}
The equations hold for $s\in [0, t]$, and
it is the case that $h_t(z)=g_t^{-1}(z)$ and $h^\rho_t(z)=\left(g^\rho_t\right)^{-1}(z)$. (However, if $t\neq s$, then
$h_s(z)\neq g_s^{-1}(z) $ and
$h^\rho_s(z)\neq \left(g^\rho_s\right)^{-1}(z)$.)

In order to obtain~\eqref{eqn:estimate-a} we will control the difference 
\begin{equation}
\abs{ \dot h_s(z)    - \dot h^\rho_s( z) },
\end{equation}
using equations~\eqref{eq:reverse-l1}--\eqref{eq:reverse-l2}
and Gr\"onwall's inequality. It is immediate that
\begin{equation}
\begin{aligned}
&\abs{ \dot h_s(z)    - \dot h^\rho_s( z) }
&\leq
\abs{\frac{1}{h_s( z)-V_1(t-s)}
-\frac{1}{h^\rho_s( z)-U^\rho_k(t-s)}
} 
+ \abs{\frac{1}{h_s( z)-V_2(t-s)}
-\frac{1}{h^\rho_s( z)-U^\rho_{k+1}(t-s)}
} \\
\label{eqn:expansion}
&+ \abs{\sum_{j\neq k, k+1} \frac{1}{h^\rho_s( z)-U^\rho_j(t-s)} }.
\end{aligned}
\end{equation}
Let $ \smallterm_{t,s}(\rho)$ denote the last term of (\ref{eqn:expansion}):
\begin{equation}\label{eqn:smallterm}
\smallterm_{t,s}(\rho)=\abs{\sum_{j\neq k, k+1} \frac{1}{h^\rho_s( z)-U^\rho_j(t-s)} }.
\end{equation}
Since $|z|\leq R$, this quantity is arbitrarily small for sufficiently large $\rho$. (The driving points other than $U^\rho_k$ and $U^\rho_{k+1}$ grow with $\rho$ (shooting out to $\pm \infty$), while $h^\rho_s(z)$ stays bounded, since $s\in [0,1]$ and $\abs z \leq R$.) Thus, it is sufficient to focus on the first term on the righthand side of (\ref{eqn:expansion}):

\begin{eqnarray}
\nonumber
\abs{\frac{1}{h_s( z)-V_1(t-s)}
-\frac{1}{h^\rho_s( z)-U^\rho_k(t-s)}
} 
\leq &&
\abs{\frac{1}{h_s( z)-V_1(t-s)}
-\frac{1}{h^\rho_s(z)-V_1(t-s)}
} \\
\label{eqn:expansion-a}
&& +\abs{\frac{1}{h^\rho_s(z)-V_1(t-s)}-\frac{1}{h^\rho_s( z)-U^\rho_k(t-s)}
}. 
\end{eqnarray}
If $\Im z\geq \hat y$, then $h_s(z)$ and $h^\rho_s(z)$ also each have imaginary part at least $\hat y$, so
\begin{eqnarray}
\nonumber
\lefteqn{\abs{\frac{1}{h_s( z)-V_1(t-s)}
-\frac{1}{h^\rho_s(z)-V_1(t-s)}
}}
\\
&&
= \frac{\abs{h_s(z)-h^\rho_s(z)}}
{\abs{\left( h_s(z)-V_1(t-s) \right)} \abs{ h^\rho_s( z)-V_1(t-s)} } 
\label{bound1}
\leq \frac{\abs{h_s(z)-h^\rho_s( z)}}
{\hat y^2}.
\end{eqnarray}

Similarly, the second term on the righthand side of \eqref{eqn:expansion-a} may be bounded:
\begin{equation}
\begin{aligned}
\abs{\frac{1}{h^\rho_s(z)-V_1(t-s)}
-\frac{1}{h^\rho_s(z)-U_k^\rho(t-s)}
} 
&=
\frac{\abs{V_1(t-s)-U^\rho_k(t-s)}}
{\abs{\left( h^\rho_s(z)-V_1(t-s) \right) \left(h^\rho_s(z)-U_k^\rho(t-s) \right)}}  \\
&\leq \frac{\abs{V_1(t-s)-U^\rho_k(t-s)}}
{\hat y^2}.
\end{aligned}
\end{equation}

A similar estimate holds for the second term on the righthand side of (\ref{eqn:expansion}), so
we conclude that
\begin{equation}\label{eqn:h_estimate_1}
\begin{aligned}
\abs{ h_s(z)    - h^\rho_s(z) } 
\leq 
&\; s\, \smallterm(\rho)
+ \int_0^s \frac{1}{\hat y^2} \Big[\abs{V_1(t-u)-U^\rho_k(t-u)} + \abs{V_2(t-u)-U^\rho_{k+1}(t-u)}\Big]
\, du\\
&+ \int_0^s \frac{2}{\hat y^2} \abs{ h_u(z)    - h^\rho_u(z) }  \,du
,
\end{aligned}
\end{equation}
where
\begin{equation}
\smallterm(\rho)=\max_{s\leq t\in [0,1]} \smallterm_{t,s}(\rho).
\end{equation}
Recall that 
Gr\"onwall's inequality says that if 
\[
\abs{f_s} < \alpha(s)+
\int_0^s \beta(u) \abs{f_u} du, 
\]
then
\[
\abs {f_s} < \alpha(s) \exp \{\int_0^s \beta(u) du \}.
\]
Let
\[
\drdist(\rho) = \max_{s\in[0,1],\, j=1,2} \abs{V_j(s)-U^\rho_{k+j-1}(s)}.
\]
Since $s\in[0,1]$, applying Gr\"onwall's inequality to \eqref{eqn:h_estimate_1} gives
\[
\abs{ h_s(z)    - h^\rho_s(z) }
\leq 
\left(\frac{\drdist(\rho)}{\hat y^2} + \smallterm(\rho) \right)\exp\{2/\hat y^2\}.
\]
The constants $\drdist(\rho)$ and $\smallterm(\rho)$ are independent of $\hat y$ and converge to $0$ as $\rho \to \infty$.
We now observe that when $s=t$, $h_s(z)=g^{-1}_t(z)$, which completes the proof of estimate~\eqref{eqn:estimate-a}.

{\em Step 3.\/} We now turn to the final step in the proof, which is the use of the Koebe distortion estimates in combination with the Gr\"{o}nwall inequalities above in order to verify the claim. Specifically, we will show that for all $\rho>\hat \rho$,
\begin{equation} \label{eqn:claim-1}
    d(\eta_1(t), \gamma^\rho_k \cup \gamma^\rho_{k+1} \cup \R)<\epsilon,
\end{equation}
and
\begin{equation} \label{eqn:claim-2}
    d(\gamma_k^\rho(t), \eta_1 \cup \eta_2 \cup \R)<\epsilon.
\end{equation}
Let $w:(0,\delta]\to \mathbb H\setminus \left(\eta_1 \cup \eta_2\right)$ be a curve with $d(w(s), \eta_1 \cup \eta_2\cup \R)=\abs{\eta_1(t)-w(s)}=s$ and $\lim_{s\to 0} w(s)=\eta_1(t)$. Denote
\[
w_s=w(s),
\]
\[
z_\delta=g_t(w_\delta),
\]
and
\[
y_\delta=\Im (z_\delta).
\]
Without loss of generality, we may assume that $w_\delta$ is in the domain of $h^\rho_t$. (If not, then \eqref{eqn:claim-1} is already satisfied.)
To verify \eqref{eqn:claim-1}, we expand

\begin{equation}
    \begin{aligned}
    d (\eta_1(t), \gamma^\rho_k \cup \R)&=d\left(\lim_{y\downarrow 0} h_t(V_1(t)+iy), \gamma^\rho_k \cup \gamma^\rho_{k+1} \cup \R \right) \\
    & \leq \abs{\lim_{y\downarrow 0} h_t(V_1(t)+iy)- h_t(z_\delta )}
    + \abs{h_t(z_\delta )- h^\rho_t (z_\delta)} 
    + d\Big(h^\rho_t (z_\delta ), \gamma^{\rho}_k \cup \gamma^\rho_{k+1} \cup \R \Big).
    \end{aligned}
\end{equation}
The first term is equal to $\delta$.
For any $\epsilon>0$, we can find $\rho$ large enough that second term is less than $\epsilon$ by Step 2 above (letting $\rho >\rho_{y_\delta, \epsilon}$).
For the third term, we apply the relationship between Euclidean distance and conformal radius and compute:
\begin{equation}\label{eqn:5.9a}
    \begin{aligned}
    d\Big(h^\rho_t (z_\delta), \gamma^{\rho}_k \cup \R \Big) 
    &\leq \crad{h^\rho_t (z_\delta ), \gamma^{\rho}_k \cup \gamma^\rho_{k+1}  \cup \R}\\
    & = \frac{2\Im(g^\rho_t( h^\rho_t (z_\delta )))}{\abs{(g^\rho_t)'\left( h^\rho_t (z_\delta ) \right)}}\\
    & = 2y_\delta \cdot \abs{(h^\rho_t)'(z_\delta) }\\
    & \leq 2 y_\delta \Big( \abs{h_t'(z_\delta)} + \abs{h_t'(z_\delta)-(h^\rho_t)'(z_\delta) } \Big)\\
    & = 2 y_\delta \abs{h_t'(z_\delta)} + 2 y_\delta  \abs{h_t'(z_\delta)-(h^\rho_t)'(z_\delta) }.
    \end{aligned}
\end{equation}
To bound the first term in \eqref{eqn:5.9a}, we will again use the relationship between Euclidean distance and conformal radius: 
\begin{equation}
    \begin{aligned}
         2 y_\delta \abs{h'_t(z_\delta)}&=\frac{2\Im\left(g_t(h_t(z_\delta) \right)}{\abs{g'_t\left(h_t(z_\delta) \right)}}\\
         &= \crad{h_t(z_\delta), \eta_1 \cup \eta_2 \cup \R }\\
        &\leq 4 \,d\left(h_t(z_\delta), \eta_1 \cup \eta_2 \cup \R \right)\\
        & =4\delta.
    \end{aligned}
\end{equation}

To bound the second term in \eqref{eqn:5.9a}, notice that the estimate (\ref{eqn:estimate-a}) may be extended to a neighborhood of $z_\delta$ by repeating the same Gr\"onwall argument as in Step 2 for $\Im(z)\geq \frac{y_\delta}{2}$.
Then differentiating Cauchy's integral formula:
\begin{equation}
\begin{aligned}
    \abs {h_t'(z_\delta)-(h^\rho_t)'(z_\delta)} 
    &= \abs{\frac{1}{2\pi i}\int_{\{w: \abs{w-z_\delta}=\frac{y_\delta }{2}\}} \frac{h_t(w)-h^\rho_t(w)}{(y_\delta/2)^2} }\\
    &\leq \frac{\max_{w:\Im(w) \geq \frac{y_\delta}{2}} \abs{h_t(w)-h^\rho_t(w)} }{y_\delta/2},
\end{aligned}
\end{equation}
so that the second term in \eqref{eqn:5.9a} is less than $4 \max_{w:\Im(w) \geq \frac{y_\delta}{2}} \abs{h_t(w)-h^\rho_t(w)} $, which is arbitrarily small for large $\rho$.

The last step is to verify \eqref{eqn:claim-2}. 

Let $\epsilon>0$, and let
\[
H(\epsilon) = \{z \in \mathbb H: \abs{z}\leq R, \text{ and } d(\eta_1 \cup \eta_2 \cup \mathbb R)\geq \epsilon \}.
\]

We will show that there exists $\rho^*$ such that 
\begin{equation}
H(\epsilon) \cap \left( \gamma^\rho_k \cup \gamma^\rho_{k+1} \cup \mathbb R \right) = \emptyset, \qquad \forall \rho >\rho^*.
\end{equation}

Since $H(\epsilon)$ is compact, there exists $\delta>0$ such that 
\[
\Im g_t(z)\geq \delta, \qquad \forall \,0\leq t \leq 1, \, \forall z \in H(\epsilon).
\]
Let $0\leq t \leq 1$ be temporarily fixed, and consider the sequence of points
\[
z_\rho=h^\rho_t(g_t(z)).
\]
Step 2 above implies that
\[
z_\rho \to z \quad \text{ as } \rho \to \infty,
\]
since we know that $\Im g_t(z) \geq \delta$ and for each $\tilde \epsilon$, there exists $\tilde\rho$ such that 
\[
\abs{h_t(w)-h^\rho_t(w)} \leq \tilde \epsilon, \qquad \forall w \text{ such that } \abs{w}\leq R, \Im w \geq \delta, \text{ and } \rho >\tilde \rho.
\]
Using the relationship between Euclidean distance and conformal radius, 
\begin{equation}\label{eqn:z_rho_lower_bound}
\begin{aligned}
d(z_\rho, \gamma^\rho_k \cup \gamma^\rho_{k+1} \cup \mathbb R) 
&\geq \frac{1}{4} \crad {z_\rho, \gamma^\rho_k \cup \gamma^\rho_{k+1} \cup \mathbb R}\\
&=\frac{\Im g^\rho_t(z_\rho)}{2 \abs{(g^\rho_t)'(z_\rho)}}\\
&=\frac{\Im g_t(z)}{2 \abs{(g^\rho_t)'(h^\rho_t(g_t(z)))}}\\
&\geq\frac{\delta}{2} \abs{(h^\rho_t)'(g_t(z))}.\\
\end{aligned}
\end{equation}
We just need to verify that there is a $\rho$-independent lower bound for $\abs{(h^\rho_t)'(z)}$.
Notice that 
\begin{equation}\label{eqn:lower_bd_h_rho_prime}
\abs {(h^\rho_t)'(g_t(z))} \geq \abs{h'_t(g_t(z))}-\abs{h'_t(g_t(z))-(h^\rho_t)'(g_t(z))}.
\end{equation}
To bound the righthand side, notice that since $z\in H(\epsilon)$,
\[
\begin{aligned}
    \epsilon &\leq \crad {z, \eta_1 \cup \eta_2 \cup \mathbb R}\\
    & =\frac{2 \Im g_t(z)}{\abs{g'_t(z)}}\\
    & =\frac{2 \Im g_t(z)}{\abs{g'_t(h_t(g_t(z)))}}\\
    & \leq 2R \abs{h'_t(g_t(z))},
\end{aligned}
\]
so 
\[
\abs{h'_t(g_t(z))} \geq \frac{\epsilon}{2R}.
\]
Again using Step 2, the second term on the righthand side of \eqref{eqn:lower_bd_h_rho_prime} is arbitrarily small for large enough $\rho$. In particular, there exists $\rho^*$ such that 
\[
\abs{h'_t(g_t(z))-(h^\rho_t)'(g_t(z))}\leq \frac{\epsilon}{4R}, \qquad \forall \rho >\rho^*, \forall z \in H(\epsilon).
\]
Plugging these bounds into \eqref{eqn:z_rho_lower_bound}, 
\[
d(z_\rho, \gamma^\rho_k \cup \gamma^\rho_{k+1} \cup \mathbb R) \geq \frac{\delta \epsilon}{8R}, \qquad \forall \rho>\rho^*.
\]
This lower bound does not depend on $\rho$, so the convergence $z_\rho \to z$ implies that
\[
d(z, \gamma^\rho_k \cup \gamma^\rho_{k+1} \cup \mathbb R) \geq \frac{\delta \epsilon}{8R}.
\]
In particular, $z\nin (\gamma^\rho_k \cup \gamma^\rho_{k+1} \cup \mathbb R)$.
In fact, the bounds are uniform for all $z\in H(\epsilon)$ and all $0\leq t \leq 1$, which verifies \eqref{eqn:claim-2}.

\end{proof}

\subsection{Proof of Theorem~\ref{maintheorem}}
We now are prepared to prove the main theorem.

\begin{proof}[Proof of Theorem~\ref{maintheorem}]
Recall that repulsion strength $\mass_t$ is given by (\ref{repulsion_strength}).

Proposition~\ref{thm:toptree} shows that the Loewner hull generated by a $\mathcal T$-indexed atomic measure with Coulombic repulsion is a graph embedding of $\mathcal T$, so what remains is checking that the repulsion strength $\mass_t$ specified in (\ref{repulsion_strength}) results in the prescribed branching angles. Most of the work is done by Theorem~\ref{thm:angles}.

In particular, for each pair $\nu, \nu'\in \mathcal T$ with a common parent $\eta=p(\nu)$, Theorem~\ref{thm:angles} specifies the relationship between the repulsion strength $\alpha_{h_{\eta}}$ and the angles between the two curves corresponding to the $\nu$ and $\nu'$ edges. In particular, there exists a time $s>h_\eta$ so that for all $t\in [h_\eta, s]$
\[
\mass_t = \frac{\pi}{\theta_\eta} -2.
\]
Then in a neighborhood of the point on the real line $U_\nu (h_\eta)$, the Loewner hull $K_{[h_\eta,s]}$ consists of two simple curves meeting the real line at angle sequence $(\theta_\eta, \pi - 2\theta_\eta, \theta_\eta)$.  

Finally, that these angles double is a result of the square-root singularity of the Loewner maps $g_t$.
\end{proof}

\begin{remark} While a marked tree $\mathcal T$ comes equipped with the graph distance, Theorem \ref{maintheorem} gives only a graph embedding of $\mathcal T$ in the halfplane, not an isometric embedding: the graph distance cannot be recovered from the hull $K_T$ alone. However, the construction encodes the graph distance in the Loewner \emph{chain} $(g_t)_{t\in [0,T]}$ via the time coordinate $t$.
 In particular, points in the marked tree $\mathcal T$ that are distance $t$ from the root are embedded as the frontier points for the Loewner evolution at time $t$.
\end{remark}

In sections \ref{sec:dyson} and \ref{sec:crt}, we will consider sequences of measures $\{\mu^n\}_{n\geq 1}$  that fall under the framework of Theorem \ref{maintheorem}. The measures will depend not on a single value of $\mass$, but on a sequence $\{ \mass_n \}$. The trees will also be given a rescaling, so that the mean lifetimes at stage $n$ are $\mean_n$. This change of the mean lifetimes is equivalent to the time change $t\mapsto \mean_n t$, so that $\frac{\mass_n}{\mean_n}$ plays the role of $\mass$ above.
Therefore, if $\mean_n=\mass_n$, then the branching angles are all $\pi/3$; in particular, the branching angles do not depend on $n$.

\section{The Dyson superprocess}
\label{sec:dyson}
We now turn to the study of superprocesses that capture scaling limits of the driving measure. There are two cases: 
\begin{enumerate}
\item {\em Dyson superprocess. \/} The scaling limit of the branching process that underlies the driving measure is the Feller diffusion.
\item {\em Conditioned Dyson superprocess. \/} The scaling limit of the genealogy of the driving measure is the continuum random tree. 
\end{enumerate}
This section is devoted to the Dyson superprocess. The conditioned case is considered in Section~\ref{sec:crt} below. 

The precise hypothesis on the genealogies of the discrete trees are stated in Theorem~\ref{mart_unconditioned} and Theorem~\ref{mart_conditioned} respectively. These hypotheses could be generalized to include other genealogies that lie within the domain of attraction of the Feller diffusion and the CRT respectively. We restrict ourselves to the simpest setting in order to illustrate the main point of our work: if one chooses branching Dyson Brownian motion as a driving measure for SLE, then the existence of limiting superprocesses follows easily from standard theory. 

The limiting superprocesses are free probability analogs of classical superprocesses. In particular, the Dyson superprocess is the free analog of the Dawson-Watanabe superprocess. 
However, the free analogs are limits of interacting particles with a singular (Coulomb) force law. Thus, uniqueness of the limit does not follow from standard theory and requires a more careful analysis using Biane's regularity theory for free convolutions~\cite{Biane}. This is a delicate problem and will be considered in separate work. 

The primary reference for this section is~\cite[Ch.1]{Etheridge}. Many of the calculations are straightforward extensions of~\cite{Etheridge}, so we will focus on the new terms that arise because of Dyson Brownian motion. 

\subsection{The martingale problem for branching Dyson Brownian motion}
\label{subsec:mp}
We first recall the existence theory for Dyson Brownian motion. For a positive integer $N$, define the Weyl chamber $W_N=\{x_1 < x_2< \ldots < x_N\} \subset \R^N$ and let $\{B_j\}_{j=1}^N$ denote standard independent Brownian motions. For $\mass>0$ and $\beta \geq 1$, there is a unique strong solution to the SDE
\be
\label{eq:dbm1}
dx_j =  \sqrt{\frac{2\mass}{\beta}}  dB_j + \mass \sum_{k\neq j} \frac{dt}{x_j-x_k},\quad  1\leq j \leq N, \quad t>0,
\ee
under the assumption that $x(0) \in \overline{W}_N$, that is
\be
\label{eq:dbm2}
x_1(0) \leq x_2(0) \leq \ldots \leq x_N(0).
\ee
Observe that particles may start on the boundary of $\overline{W}_N$, but are immediately driven into the interior (see~\cite[Thm.4.3.2]{Anderson_Guionnet_Zeitouni}). The strength of the interaction is usually standardized to $1/N$. We have replaced it by $\mass$ since  the number of particles  changes with branching. 

Let $\theta$ be distributed as a critical binary Galton-Watson forest such that the lifetimes of individuals are iid exponential random variables with mean $\mean$. Let $p(\nu)$ denote the parent of an individual $\nu \in \theta$, let $l_\nu$ denote its lifetime (i.e. the time between its birth and death), and define the markings $h_\nu = h_{p(\nu)}+l_\nu$ inductively from the root. Let $B_\nu(t)$, $\nu \in \theta$ denote standard independent Brownian motions indexed by $\theta$. Assume given an initial condition $a \in \overline{W}_{|\theta_0|}$ where $|\theta_0|$ denotes the number of individuals in $\theta_0$. 

Roughly speaking, we may construct branching Dyson Brownian motion by applying the above existence theory for Dyson Brownian motion on time intervals between birth-death times. More precisely, we define branching Dyson Brownian motion $\{x_\nu(t)\}_{\nu \in \theta_t}$, $t\geq 0$ with initial condition $a$ to be the spatial branching process that is the unique strong solution to 
\begin{equation}
\label{Dysonrevised}
\begin{aligned}
d x_\nu (t)&= \sqrt{\frac{2\mass}{\beta}} \,dB_\nu (t) +  \mass \sum_{\substack{\eta\in \theta_t \\ \eta\neq \nu}}\frac{1}{x_\nu (t)-x_\eta(t)}\,dt, \quad t \in (h_{p(\nu)},h_\nu),\\
x_\nu \left(h_{p(\nu)}\right)&=\lim_{t\nearrow h_{p(\nu)}} x_{p(\nu)}(t), \quad x_\nu(0)=a_\nu, \quad \nu \in \theta_0,
\end{aligned}
\end{equation}
Strong existence for \qref{Dysonrevised} conditional on branching Brownian motion may be constructed using Perkins stochastic calculus.

Informally, these equations express the fact that the points $\{x_\nu(t)\}_{\nu \in \theta}$ solve Dyson Brownian motion in the time intervals between birth-death times and that the spatial location of each individual $x_\nu$ at its birth time is determined by the spatial location of its parent $x_{p(\nu)}$ at its death time $h_{p(\nu)}$. 

Rather than treat the genealogy as given, it is convenient to formulate \qref{Dysonrevised} as a martingale problem, since this fits naturally with both the Loewner theory and scaling limit.  To this end, we define the purely atomic measures
\be
\label{eq:dbm3}
\xi_t= \sum_{\nu \in \theta_t} \delta_{x_\nu (t)}, \quad \xi_0 = \sum_{\nu \in \theta_0}\delta_{a_\nu}.
\ee
We pair the random variable $\xi_t$ with positive test functions $\psi \in C_b^+(\R) \cap C^2(\R)$ as follows. Define
\be
\label{eq:gen4}
F(\xi_t)=\exp{\langle \log \psi, \xi_t \rangle} = \prod_{\nu \in \theta_t} \psi\left(x_\nu(t)\right).
\ee
Given a measure $\mu$ and a test function $\phi \in C_b^+(\R) \cap C^2(\R)$, the following quadratic form plays a basic role in the study of Dyson Brownian motion and its limits
\be
\label{eq:gen2}
H_{\mu}(\phi) := \frac{1}{2}\int_\R\int_\R \frac{\phi'(x)-\phi'(y)}{x-y} \mu(dx)\mu(dy).
\ee
When considering discrete approximations, we will need to remove the contribution on the diagonal $x=y$, so we also define
\be
\label{eq:gen2b}
\newG_{\mu}(\phi) := \frac{1}{2}\int_\R\int_\R \frac{\phi'(x)-\phi'(y)}{x-y} \mu(dx)\mu(dy) -\frac{1}{2}\langle \phi'', \mu \rangle.
\ee
\begin{theorem}
\label{mart_unconditioned}
The measure $\xi_t$ solves the $(\mathcal L, \xi_0 )$ martingale problem for 
\be
\label{generator}
\mathcal L F(\xi_t) = F(\xi_t) \left( \mass \newG_{\xi_t}(\log \psi) +
\frac{\mass}{\beta} \left\langle \frac{\psi''}{\psi}, \xi_t \right\rangle   + 
\mean \left \langle \frac{1}{2}\left(\psi + \frac{1}{\psi}\right)-1, \xi_t \right\rangle \right).
\ee
\end{theorem}
By the definition of the martingale problem~\cite{Etheridge,SV}, what this means is that
the random variable $\xi \in D_{\mathcal M_F (\mathbb R)} [0,\infty)$ has the property that 
\begin{equation}
\label{eq:gen2a}	
F(\xi_t)-F(\xi_0) -\int_0^t \mathcal L F(\xi_s) ds
\end{equation}
is a mean zero $\mathbb P_{\xi_0}$ martingale for all $F$ in the domain of the generator $\mathcal L$. The class of test functions $F$ defined in~\qref{eq:gen4} using $\psi \in C_b^+({\mathbb R}) \cap C^2(\R)$ is a core for the generator $\mathcal{L}$.
Most of the proof of Theorem~\ref{mart_unconditioned} reduces to computing   $\mathcal L$.
\begin{lemma}\label{mart_unconditioned_lemma}
For $\mathcal L$ defined as in Theorem \ref{mart_unconditioned},
\begin{equation}
\frac{d}{dr} \mathbb E_{\xi_0} [F(\xi_r)]\Big\vert_{r=t}=\mathbb E_{\xi_0} [\mathcal L F(\xi_t)], \quad t>0.
\end{equation}
\end{lemma}

\begin{proof}[Proof of Lemma \ref{mart_unconditioned_lemma}]
The form of the generator can be understood by first ignoring branching and focusing on spatial motion. To this end, we 
apply \Ito's\/ formula to $F(\xi_t) = \prod_{l}\psi(x_l(t))$ where $x(t)$ is Dyson Brownian motion in the form  \qref{eq:dbm1}. Then
\be
dF = \sum_j \frac{\partial F}{\partial x_j}dx_j + \frac{1}{2}\sum_{j,k}\frac{\partial^2 F}{\partial x_j \partial x_k} dx_j dx_k= \sum_j \frac{\partial F}{\partial x_j}dx_j  + \frac{\mass}{\beta} \sum_{k} \frac{\partial^2 F}{\partial x_k^2} dt.
\ee
We also have the identities
\be
\label{eq:dbm4}
\frac{\partial F}{\partial x_j} =  F \frac{\psi'(x_j)}{\psi(x_j)} = F (\log\psi)'(x_j), \quad \frac{\partial^2 F}{\partial x_k^2} =  F \frac{\psi''(x_k)}{\psi(x_k)}.
\ee
We use \qref{eq:dbm1} to see that the first of these terms gives the drift 
\ba
\nn
\lefteqn{\mass F \sum_{j} (\log\psi)'(x_j) \sum_{k\neq j} \frac{1}{x_j-x_k} }\\
\label{eq:gen6}
&& =  \frac{\mass F}{2}\sum_{j \neq k }\frac{(\log\psi)'(x_j)-(\log\psi)'(x_k)}{x_j-x_k} = \mass F(\xi_t)\newG_{\xi_t}(\log \psi).
\ea
where we have used the definition of $\xi_t$ and equation \qref{eq:gen2b} in the second equation. In an analogous manner, the second term in \qref{eq:dbm4} gives rise to
\begin{equation}
\label{eq:gen7}
\frac{\mass}{\beta} F(\xi_t) \left\langle \frac{\psi''}{\psi}, \xi_t \right\rangle.  
\end{equation}
Equations~\qref{eq:gen6} and~\qref{eq:gen7} give the first two terms in equation~\qref{generator}. 

To complete the calculation, we need to consider these terms in combination with branching. The branching mechanism for $\xi_t$  is that of branching Brownian motion, so this is a standard conditioning argument. We omit the details (see the proof of Theorem 1.4 in~\cite{Etheridge} for example or the proof of Theorem~\ref{mart_conditioned} below).
\end{proof} 

\begin{proof}[Proof of Theorem \ref{mart_unconditioned}]
We follow \cite[\S 1.2]{Etheridge}. Integrating the result of 
Lemma~\ref{mart_unconditioned_lemma}:
\begin{equation}
\mathbb E_{\xi_0} [F(\xi_{t+u})-F(\xi_t))]=\mathbb E_{\xi_0} \left[\int_t^{t+u} \mathcal L F(\xi_s) ds\right].
\end{equation}
By the Markov property,
\begin{equation}
\nn
\mathbb E_{\xi_0} [F(\xi_{t+u})-F(\xi_t)\vert \mathcal F_t] =\mathbb E_{\xi_t} \left[\int_0^u \mathcal L F(\xi_s) ds\right] = \mathbb E_{\xi_0} \left[\int_t^{t+u} \mathcal L F(\xi_s) ds \vert \mathcal F_t\right],
\end{equation}
which proves that 
\begin{equation}
\nn
F(\xi_t)-F(\xi_0) -\int_0^t \mathcal L F(\xi_s) ds
\end{equation}
is a mean zero $\mathbb P_{\xi_0}$ martingale for all $F$ of the form~\qref{eq:gen4}.
\end{proof}

\begin{lemma}\label{semimartingale_unconditioned}
For each $\phi \in C^+_b(\mathbb R) \cap C^2 (\R)$, the process 
$\{\left \langle \phi, \xi_t\right \rangle\}_{t \geq 0}$  is a semimartingale that is the sum of a predictable finite variation process 
\begin{equation}\label{martingale_k}
V_t(\phi)= \langle \phi, \xi_0\rangle
+ \mass \int_0^t\left(H_{\xi_s}(\phi) + \left(\frac{1}{\beta}- \frac{1}{2}\right)\left \langle  \phi'', \xi_s\right\rangle \right) \, ds,
\end{equation}
and a martingale, $M_t(\phi)$, with quadratic variation
\begin{equation}\label{quadratic_var_k}
[M(\phi)]_t= 2\int_0^t \left \langle  \frac{\mass}{\beta} \left(\phi' \right)^2 + \mean (\cosh \phi -1), \xi_s  \right \rangle ds.
	\end{equation}
\end{lemma}
This lemma is the analog of~\cite[Lemma 1.10]{Etheridge}, the main difference being the term $H_{\xi_t}(\phi)$. We include the proof since we will need it to establish tightness.

\begin{proof}
We consider $\theta\geq 0$, $\psi=e^{-\theta \phi}$ and substitute the test function $F_\theta(\xi_t)=\exp{\langle -\theta \phi, \xi_t\rangle}$ in equation~\qref{eq:gen2a}.
We then take the expectation of \qref{eq:gen2a}, differentiate  with respect to $\theta$, and evaluate at $\theta=0$ to obtain
\begin{equation}
\label{eq:diffgen1}	
0=\mathbb E_{\xi_0} \left[ \langle \phi, \xi_{t+u} \rangle -\langle \phi, \xi_t\rangle 
+\int_t^{t+u} \frac{d}{d\theta} \mathcal L F_\theta (\xi_s)\big\vert_{\theta=0}
ds  \right].
\end{equation}
We compute the integrand using Theorem~\ref{mart_unconditioned}: 
\ba
\nn
\lefteqn{\mathcal L  F_\theta(\xi_s)
= \exp\left\langle -\theta \phi,\xi_s \right \rangle
\times}\\
\label{eq:diffgen2}	
&& \left( -\theta \mass \newG_{\xi_s}(\phi) + \left \langle \frac{\mass}{\beta}\left(\theta^2 (\phi')^2 - \theta \phi'' \right)+ \mean(\cosh \theta \phi -1 ), \xi_s\right\rangle\right).
\ea
Now differentiate with respect to $\theta$ and use the definition of $H$ in equation~\qref{eq:gen2} to obtain
\begin{equation}
\label{eq:diffgen3}	
\frac{d}{d\theta} \mathcal L F_\theta (\xi_s)\big\vert_{\theta=0}= 
- \mass \left( H_{\xi_s}(\phi) +
\left(\frac{1}{\beta}- \frac{1}{2}\right)
\left \langle \phi'', \xi_s \right\rangle\right).
\end{equation}	
We substitute equation \qref{eq:diffgen3} in \qref{eq:diffgen1} to conclude that $\langle \phi, \xi_t\rangle$ is a semimartingale with a predictable finite variation process $V_t(\phi)$ given in \qref{martingale_k}.  

The quadratic variation of $M_t(\phi)$ is computed as follows. Since $\langle \phi, \xi_t\rangle$ is a semimartingale,  we apply It\^o's formula backwards to conclude that
\begin{equation}
\nn
\begin{aligned}
\exp\langle- \phi, \xi_t\rangle-\exp\langle -\phi, \xi_0\rangle
	&+ \int^t_0 \exp\langle -\phi, \xi_s\rangle dV_s 
	-\frac{1}{2}\int^t_0 \exp\langle -\phi, \xi_s\rangle d[M(\phi)]_s
	\end{aligned}
	\end{equation}
	is a martingale. On the other hand, letting $\theta =1$ in \qref{eq:diffgen2}, and using \qref{martingale_k}, Theorem \ref{mart_unconditioned} implies that
\begin{equation}
\nn
\exp\langle- \phi, \xi_t\rangle-\exp\langle -\phi, \xi_0\rangle
+ \int^t_0 \exp\langle -\phi, \xi_s\rangle dV_s -\int_0^t \exp\langle- \phi, \xi_s\rangle \left \langle  \frac{\mass}{\beta} (\phi')^2 + \mean (\cosh \phi -1), \xi_s  \right \rangle ds,
\end{equation}
is a martingale. We compare the above equations to obtain \qref{quadratic_var_k}.
\end{proof}

\subsection{Parameters for the scaling limit} 
\label{subsec:dyson}
We now apply the results of the previous section to a family of branching Dyson Brownian motions, $\{\xi_t^n\}$, indexed by positive integers $n$. 
We establish tightness and characterize subsequential limits for the rescaled processes 
\be
\label{eq:scaling2}
\mu^n_t = \frac{1}{n}\xi^n_{t}, \quad t \geq 0,
\ee
under the assumptions that
\begin{enumerate}
\item[(a)] The  initial data $\{\mu^n_0\}_{n\geq 1}$ is tight and converges weakly to a measure $\mu_0\in \mathcal{M}_F(\R)$. (Here and below we extend the notation of Section~\ref{subsec:mp} in the obvious manner with sub- or superscripts $n$.)
\item[(b)] $\theta_n$ is a critical binary Galton-Watson forest with branching rate $\mean_n=n$ (equivalently mean lifetime $1/n)$ and an initial population of size $\langle 1, \mu^n_0\rangle$.	
\item[(c)] The parameters of the Dyson Brownian motion \qref{Dysonrevised} satisfy
\be
\label{eq:scaling1}
\mass_n = \frac{1}{n}, \quad \beta_n = \beta \geq 1.
\ee
\end{enumerate}

The rescaling assumption~\qref{eq:scaling2}, along with the time rescaling implicit in (b), is in accordance with the classical theory of continuous state branching processes. In particular, the total mass $N^n_t = \langle 1, \mu^n_t\rangle$ converges to the Feller diffusion with initial mass $\langle 1, \mu^n_0\rangle$. Since the total population size is $O(n)$ it is then necessary to choose $\mass_n=O(1/n)$, so that the Coulombic repulsion in Dyson Brownian motion has a nontrivial limit. This is why (b) takes the form it does. Finally, the restriction $\beta \geq 1$ is needed for the well-posedness of Dyson Brownian motion. We hold $\beta$ fixed as is standard in random matrix theory. Formally, $\beta$ plays no role in the superprocess limit below, though it does significantly affect the conformal process for any finite $n$ and it should be expected to significantly influence fluctuations from the limit.

These rescalings have the following effect on the terms introduced in the last section. We fix a test function $\phi \in C_b(\R)\cap C^2(\R)$ and consider the test functions
\be
\label{eq:newgen1}
\psi_n(x) = \exp \left(\frac{\phi}{n}\right), \quad  \newF(\mu^n_t):= \exp \langle \phi, \mu^n_t \rangle = \exp \langle \log \psi_n, \xi^n_t \rangle = F(\xi^n_t).
\ee
We substitute these parameter choices in equation~\qref{generator} to obtain the generator
\ba
\nn
\lefteqn{\newgen^n \newF (\mu^n_t) =\newF(\mu^n_t)\times}\\
\label{eq:newgen2}
&& \left(H_{\mu^n_t}(\phi) +\frac{1}{n}\left(\frac{1}{\beta}-\frac{1}{2}\right) \left\langle \phi'', \mu^n_t\right\rangle + \frac{1}{n^2\beta} \left\langle (\phi')^2,  \mu^n_t\right\rangle + n^2  \left\langle \cosh \frac{\phi}{n}-1, \mu^n_t\right\rangle \right).
\ea
Note that $H$ (defined in~\qref{eq:gen2}) is quadratic in $\mu$, whereas the contribution on the diagonal, which is subtracted away in~\qref{eq:newgen2}, is of lower order. 

In a similar manner, using $-\theta \phi$ instead of $\phi$ in \qref{eq:newgen1},  differentiating with respect to $\theta$ and evaluating at $\theta=0$ as in Lemma~\ref{semimartingale_unconditioned}, we obtain the semimartingales
\be
\label{eq:newmart1}
\langle \phi, \mu^n_t \rangle = V^n_t + M^n_t, \quad n \geq 1,
\ee
with the predictable finite variation process 
\be
\label{eq:newmart2}
V^n_t =V^n_t (\phi)= \langle \phi, \mu^n_0 \rangle + \int_0^t \left( H_{\mu^n_s}(\phi)\,  + \frac{1}{n} \left(\frac{1}{\beta}-\frac{1}{2}\right)\left\langle \phi'', \mu^n_s\right\rangle \right) \, ds,
\ee
and a martingale $M^n_t$ with quadratic variation
\be
\label{eq:newmart3}
[M^n]_t = [M^n]_t (\phi) =2\int_0^t \left\langle \frac{1}{n^2\beta} (\phi')^2 + n^2\left(\cosh \frac{\phi}{n}-1\right), \mu^n_s \right \rangle \, ds.
\ee

The above calculations immediately reveal the nature of the scaling limit. If $\lim_{n\to \infty} \mu^n_{\cdot} =\mu_{\cdot}$ in $\mathcal{D}([0,\infty),\mathcal{M}_F(\R))$, then $\newF(\mu^n_t)\to \newF(\mu_t)$ for all $t \in [0,\infty)$ and the generator of the limiting process is
\be
\label{eq:newgen3}
\newgen \newF (\mu_t) = \newF(\mu_t) \left(H_{\mu_t}(\phi) + \frac{1}{2} \left\langle \phi^2, \mu_t\right\rangle \right).
\ee
Similarly, the limiting semimartingale is $\langle \phi,\mu_t \rangle=V_t +M_t$ where
\be
\label{eq:newmart4}
V_t = \langle \phi, \mu_0 \rangle + \int_0^t H_{\mu_s}(\phi)\, ds, \quad [M]_t = \int_0^t \left\langle \phi^2,\mu_s \right\rangle ds.
\ee
We now turn to a rigorous analysis of this limit. We use the semimartingales to establish tightness.

\subsection{Tightness}
\label{subsec:tight}
The proof of tightness parallels that for the Dawson-Watanabe superprocess ~\cite[\S 1.4]{Etheridge}. Tightness of the processes $\{\mu^n\}_{n \geq 1}$
in the Skorokhod space $D([0,\infty),\mathcal{M}_F(\R))$ is established by combining the Aldous-Rebolledo criterion for tightness of real-valued semimartingales with a compact containment condition that allows us to extend this condition to measure-valued processes. Once this `infrastructure' has been developed, as in~\cite[\S 1.4]{Etheridge}, a few simple calculations are all that is needed to obtain tightness. 

The main new observation we need, which is immediate from~\qref{eq:gen2b}, is that for any finite positive measure $\mu$ and test function $\phi \in C^2(\R)$ we have the estimate
\be
\label{eq:scaling3}
\left| H_\mu (\phi)\right| \leq \frac{1}{2}\|\phi''\|_{\infty}\langle 1,\mu\rangle^2 .
\ee

	For each $n\geq 1$, let $\theta^n$ be distributed as a critical binary Galton-Watson forest with independent exponential lifetimes of mean $\mean_n$ starting with $n$ individuals. Let $\{\mu^n\}_{n\geq1}$ be the sequence of measure-valued processes satysfying (\ref{Dysonsystem}) with parameter $\mass_n$. 

\begin{theorem}\label{tightness_unconditioned}
	If $\mass_n=\mean_n=\frac{1}{n}$ and the sequence of initial configurations $\{ \mu^n_0\}=\{ \sum_{\nu \in \mathcal \theta^n_0} \delta_{U_\nu(0)} \}$ is tight in $\mathcal M_F(\mathbb R)$,
	then the sequence $\{\mu^n\}_{n\geq 1}$ is tight in $D_{\mathcal M_F(\hat {\mathbb R})}[0,\infty)$.
\end{theorem}

\begin{proof}[Proof of Theorem \ref{tightness_unconditioned}]
Our task reduces to checking compact containment and the Aldous-Rebolledo Criterion as in~\cite[Prop. 1.19]{Etheridge}. Let $\hat {\mathbb R}$ denote the compactification of $\mathbb R$. (After verifying tightness in the compactified space  $D_{\mathcal M_F(\hat{\mathbb R})}[0,\infty)$, we will show that any subsequential limit actually is contained in $D_{\mathcal M_F(\mathbb R)}[0,\infty)$.) 

Compact containment is straightforward. As noted above, the rescaled total mass process $\langle 1, \mu^n_t\rangle$ converges to the Feller diffusion with initial number $\langle 1, \mu_0\rangle$. Since $\langle 1, \mu_t^n\rangle$ is a martingale, for any $T>0$ and $K>0$
\be
\PP \left( \sup_{0 \leq t \leq T} \langle 1, \mu_t^n\rangle > K    \right) \leq \frac{1}{K}\langle 1, \mu_0^n \rangle. 
\ee
The right hand side tends to $0$, uniformly in $n$, as $K \to \infty$ because of our assumption that the initial number $\langle 1, \mu^n_0\rangle \to \langle 1, \mu_0\rangle$.

We now check the Aldous-Rebolledo criterion \cite{Ald78, Reb80} for $\langle \phi, \mu^n_t \rangle$. Although we are temporarily working on $\hat \R$, we may continue to use test functions $\phi \in C_b^+(\mathbb R) \cap C^2(\mathbb R)$, since the $\mu^n_s$ have no mass at $\infty$.
The calculation above shows that the real-valued sequence $\langle \phi, \mu^n_t \rangle$ is tight for each $\phi$ and fixed $t>0$. Each term in this sequence is a semimartingale with decomposition given by (\ref{eq:newmart1}). 
Following the treatment in \cite{RC86}, we only need to verify that given a sequence of stopping times $\tau_n$, bounded by $n$, for each $\epsilon>0$ there exists $\delta>0$ and $n_0$ such that
\begin{equation}\label{V_unconditioned}
\sup_{n\geq n_0} \sup_{\eta\in [0, \delta]} \mathbb P\left[ \abs{V^n _{\tau_n + \eta} -V^{n}_{\tau_n } } >\epsilon  \right] \leq \epsilon,
\end{equation}
and 
\begin{equation}\label{M_unconditioned}
\sup_{n\geq n_0} \sup_{\eta\in [0, \delta]} \mathbb P\left [ \abs{ [M^{n}]_{\tau_n + \eta} - [M^{n}]_{\tau_n} } >\epsilon  \right] \leq \epsilon.
\end{equation}
where $V^n_t$ and $M^n_t$ are defined in equations~\qref{eq:newmart2} and \qref{eq:newmart3}. 

We always have $\abs{\beta^{-1}-2^{-1}}\leq 2^{-1}$.  We use \qref{eq:newmart2},  \qref{eq:scaling3} and the fact that $\langle 1, \mu^n_t\rangle$ is a martingale to see that for $0<\eta< \delta$
\begin{equation}
\begin{aligned}
\nn	
\mathbb E \left( \abs{V^n_{\tau_n + \eta} -V^{n}_{\tau_n}} \right) 
& \leq \frac{1}{2}\norm{\phi''}_\infty \mathbb E \left( \int_{\tau_n}^{\tau_n+\eta}  \langle 1, \mu^n_s \rangle^2 
+ \frac{1}{n} \langle 1, \mu^n_s \rangle ds \right)\\
& = \frac{1}{2}\norm{\phi''}_\infty \mathbb E \left( \mathbb E_{\mu^n_{\tau_n}} \left( \int_0^\eta \langle 1, \mu^n_s \rangle^2 
+ \frac{1}{n} \langle 1, \mu^n_s \rangle ds \right) \right)\\
&\leq  \delta \norm{\phi''}_\infty  \mathbb E \left(\langle 1, \mu^n_{\tau_n} \rangle^2 
+  \frac{1}{n}\langle 1, \mu^n_{\tau_n} \rangle \right)\\
& \leq  \delta \norm{\phi''}_\infty\left( \langle 1, \mu^n_0 \rangle^2 
+ \frac{1}{n}\langle 1, \mu^n_0 \rangle\right).
\end{aligned}
\end{equation}
The term $\langle 1, \mu^n_0 \rangle$ is uniformly bounded in $n$ by our assumption on tightness of the initial data, and we have used that $\norm{\phi''}_\infty$ is finite. The estimate~\qref{V_unconditioned} now follows from Chebyshev's inequality.

A similar argument applies to $[M^n(\phi)]_t$. Since $\|\phi\|_\infty <\infty$ we have the estimate
\be
\label{eq:scaling4}
\sup_{n \geq 1}\abs{n^2 \left( \cosh \frac{\phi}{n}-1\right)} \leq C <\infty,
\ee	
where the constant $C$ depends only on $\|\phi\|_\infty$. Further, every $\phi \in C_b^+(\R)\cap C^2(\R)$ can be uniformly approximated by $\phi$ with $\||\phi'\|_\infty <\infty$.  An argument as above now shows that 
$[M^n(\phi)]_t$ satisfies (\ref{M_unconditioned}). 
\end{proof}
	
\subsection{The martingale problem for the Dyson superprocess}	

Theorem \ref{LimitingMartingaleProblem} states that every subsequential limit lies in $\mathcal M_F(\mathbb R)$ (not merely in $\mathcal M_F(\hat{\mathbb R})$) and solves a particular martingale problem. In order to prove the theorem we first show that every subsequential limit of $\mu^n$ in $\mathcal M_F(\hat{\mathbb R})$ satisfies a martingale problem. We then show that the subsequential limit lies in  $\mathcal M_F(\mathbb R)$. 
\begin{proof}[Proof of Theorem \ref{LimitingMartingaleProblem}, convergence of generators]
Suppose $\mu^{n_k}_\cdot$ is a subsequence that converges to a limit $\mu_\cdot$ in $\mathcal{D}\left([0,\infty),\mathcal M_F(\hat{\mathbb R})\right)$. By relabeling the sequence, we may refer to it simply as $\mu^n_\cdot$. 

In order to show that $\mu_{\cdot}$ satisfies the martingale problem associated to $\newgen$ it is enough to show that (see~\cite[\S 3, eqn (3.4)]{Ethier&Kurtz})
\be
\label{eq:scaling6}
0 = \mathbb{E}\left[ \left( \newF(\mu_{t_{p+1}})- \newF(\mu_{t_{p}}) - \int_{t_p}^{t_{p+1}} \newgen \newF(\mu_s)\right) \prod_{j=1}^p h_j(\mu_{t_j}) \right]
\ee
whenever $0 \leq t_1 < t_2 < \ldots < t_{p+1}$, $\newF$ lies in the domain of $\newgen$ and $\{h_j\}_{j=1}^p$ are bounded measurable functions on $\mathcal{M}_F(\hat{R})$. 

Let $\phi \in C_b(\R)\cap C^2(\R)$. We have computed the action of the generator $\newgen_n$ on test functions $\newF(\mu^n_t)= \exp\langle \phi,\mu^n_t\rangle$ in equation~\qref{eq:newgen2}. Further, this class of test functions is a core for the generator $\newgen_n$. Theorem~\ref{mart_unconditioned} therefore implies that the analog of equation \qref{eq:scaling6} holds for every $n$ (replacing $\mu$ with $\mu^n$ and $\newgen$ with $\newgen_n$). 

The estimates \qref{eq:scaling3} and \qref{eq:scaling4}, along with equations \qref{eq:newgen2} and \qref{eq:newgen3}, immediately show that for every fixed $t \in [0,\infty)$
\begin{equation}
\label{eq:scaling5}
\lim_{n \to \infty } \newgen_n \newF(\mu^n_t) =  \newgen \newF(\mu_t).
\end{equation}
For every finite collection of times $0 \leq t_1 < t_2 < \ldots < t_{p+1}$ and every collection of bounded functions $\{h_j\}_{j=1}^p$, we may use the bounded convergence theorem to establish \qref{eq:scaling6}. 
\end{proof}

\begin{proof}[Proof of Theorem \ref{LimitingMartingaleProblem}, compact containment]
Finally, to see that the limit is in $\mathcal M_F(\mathbb R)$,
we must show that 
\begin{equation}
\lim_{R\to \infty}\mathbb P \left[ \mu_t(\{\abs x \geq R \})>\epsilon \right]=0.
\end{equation}
However, since the indicator function $\mathbbm 1_{\abs{x}\geq R}$ is not a permissible test function, we use test functions $\phi_R\in C^+_b(\hat {\mathbb R})\cap \mathcal D(\Delta)$ satisfying 
\begin{equation}
\begin{aligned}
    \norm {\phi''_R}_\infty &\leq 2/R\\
    \phi_R(x) = 1, \quad &\abs x \geq R\\
		\phi_R(x)  = 0, \quad & \abs x \leq \frac{R}{2}.
\end{aligned}
\end{equation}
The calculation above implies that
		\begin{equation}
		\begin{aligned}\label{unifbound?}
		\mathbb E [\langle \phi_R, \mu_t\rangle] 
		& = \langle \phi_R, \mu_0\rangle + \mathbb E \left[ \frac{1}{2}\int_0^t \int_{\mathbb R }\int_{\mathbb R} \frac{\phi_R'(x)-\phi_R'(y)}{x-y} \mu_s(dx) \mu_s(dy) ds \right].\\
		\end{aligned}
		\end{equation}
	Since \[\abs{\frac{\phi_R'(x)-\phi'_R(y)}{x-y}} \leq \frac{2}{R},\] we compute 
		\begin{equation}
		\begin{aligned}
		\mathbb E [\langle \mathbbm 1_{\abs{x}>R }, \mu_t\rangle] &\leq \mathbb E [\langle \phi_R, \mu_t\rangle]\\
		& \leq \langle \phi_R, \mu_0\rangle + \mathbb E \left[ \frac{1}{2} \int_0^t \int_{\mathbb R}\int_{\mathbb R} \frac{\phi_R'(x)-\phi_R'(y)}{x-y} \mu_s(dy) \mu_s(dx) ds \right]\\
		& \leq  \langle \phi_R, \mu_0\rangle +  \frac{1}{R}\mathbb E \left[ \int_0^t \langle 1, \mu_s\rangle ^2 ds\right].
		\end{aligned}
		\end{equation}
		The term $\langle \phi_R, \mu_0\rangle$ equals $0$ for sufficiently large $R$, since $\mu_0$ has compact support. Applying Markov's inequality finishes the proof.
	\end{proof}

\section{The conditioned Dyson superprocess}
\label{sec:crt}
We now consider the scaling limit of the driving measure in a different setting: the branching structure is now a tree conditioned to converge to the continuum random tree. The spatial motion remains the same as in \S\ref{sec:dyson} and satisfies (\ref{eq:dbm1}).

Mirroring \S\ref{sec:dyson}, we denote by $\hat \xi^n$ the time-dependent purely atomic measure 
\be
\hat \xi^n_t = \sum_{\nu \in \hat \theta^n_t} \delta_{x_\nu(t)}, \quad \hat \xi^n_0=\delta_0,
\ee
where the location $x_\nu$ of each individual is the unique strong solution to (\ref{Dysonrevised}), with branching structure given by $\hat \theta^n$ instead of $\theta$. Here $\hat \theta^n$ is a critical binary Galton-Watson tree with branching rate $m_n$ conditioned to have $n$ edges. The initial condition is a single Dirac delta mass at $0$.

Results for the conditioned process are analogous to the results of \S\ref{sec:dyson} for the unconditioned process, but we use a $n^{-1/2}$ rescaling, and additional care must be taken with the branching term of the generator.


Unlike in the unconditioned case covered in \S\ref{sec:dyson}, conditioning on total population causes the offspring distribution to vary with each step. 
Here, the expected number of offspring at time $t$ will be $2 \Qk{t}$ (defined below in Equation (\ref{eqn:exp_offspring})), whereas in the unconditioned case the expected number of offspring is always equal to one. Replacing $\Qk{t}$ with $\frac{1}{2}$ throughout this section recovers the results of \S\ref{sec:dyson}.

All notation from \S\ref{sec:dyson} continues to be in effect, though we make more frequent use of the sub- and superscripts $n$ (since the conditioning requires it), and often use a $\,\hat{}\,$ to distinguish a conditioned random variable in this section from its counterpart in \S\ref{sec:dyson}.
We will use $\mathbb E^n_{\xi}$ to denote expectation when the initial condition is $\xi$ and the total population is conditioned to be $n$. Furthermore, the conditioned Galton-Watson process will be denoted by $\hat N^n_t=\abs{\theta^n_t} = \langle 1, \hat \xi^n_t \rangle$. 

\subsection{Parameters for the scaling limit and the generator for conditioned branching}
The scaling limit in \S\ref{sec:dyson} was obtained by choosing branching rate $m_n=n$ and rescaling the total mass process by $n^{-1}$ so that $\langle 1, n^{-1} \xi^n_t \rangle$ converged in distribution to the Feller diffusion. In contrast, here we choose branching rate $m_n=2 \sqrt n$ and rescale by $n^{-1/2}$. In this case, the total mass process converges to the total local time at level $t$ of the normalized Brownian excursion, which we denote by $\brexlocal^t:=\brexlocal^t(1)$ \cite{Pitmanbook}:
\begin{equation}
n^{-1/2}\langle 1, \hat \xi^n_t \rangle \distr\frac{ \brexlocal^t}{2},
\end{equation}
where the local time process is defined by
\be \label{def:localtime}
\brexlocal^t(\tau)=\lim_{\epsilon \downarrow 0} \frac{1}{2\epsilon} \int_0^\tau 1_{\{t-\epsilon <B_s<t+\epsilon\}} \;ds.
\ee

To compute probabilities, we use the following comparison between the supremum of the local time and the supremum of the reflected Brownian bridge (see \cite{PitmanLocalTime}, equation (35)):
\begin{equation}\label{LocalTimeFormula}
\sup_{t\geq 0} \brexlocal^t\dist 4 \sup_{0\leq t \leq 1} B^{\abs{\text{br}}, 1}_t,
\end{equation}
where $B^{\abs{\text{br}}, 1}$ is the reflected Brownian bridge of length $1$. 

The most significant difference between the conditioned and unconditioned cases is the generator for the branching term. If $\hat \theta^n$ is a critical binary Galton-Watson process with branching rate $m_n$ conditioned to have total population $n$, then the branching mechanism is nonhomogeneous in time, so for each $n$, the conditioned process $\hat N^n_t$ is not a Markov process.  However, by keeping track of the number of ``remaining'' individuals via the random variable $R^n_t$, the process $(\hat N^n_t, R^n_t)$ is a Markov process, and the generator for the branching term is similar in form to the unconditioned case.

More precisely, define
\begin{equation} 
R^n_t=n-\#{\{\nu \in \hat \theta^n: l(\nu)<t\}},
\end{equation}
and let $\Qk{t}$ denote the probability that a death at time $t$ results in two offspring (rather than $0$). This probability satisfies
\begin{equation}\label{eqn:exp_offspring}
 \Qk{t}= \frac{(\hat N^n_t+1)(R^n_t-\hat N^n_t)}{2\hat N^n_t(R^n_t-1)}.
\end{equation}
Define
\begin{equation}\label {extra_term}
\mathcal G^n F(\xi_t)=F(\xi)\left \langle  \mean_n \left(  \Qk{t}\,\psi - 1 + (1-\Qk{t})\frac{1}{\psi}\right), \xi_t\right \rangle.
\end{equation}
Notice that if $\Qk{t}=1/2$, this is exactly the final term in (\ref{generator}), since in the unconditioned discrete process, the probability of having $2$ or $0$ offspring is always $1/2$.

\subsection{The martingale problem for conditioned branching Dyson Brownian motion}
Mirroring the results for the unconditioned case, we show that the process $\hat \xi^n_t$ solves a martingale problem.

For $n\geq 1$, let $\hat \theta^n$ be distributed as a critical binary Galton-Watson tree with branching rate $\mean_n$, conditioned to have $n$ edges. Let
$\{\hat\xi^n\}_{n\geq 1}$ be the sequence of random variables taking values in $D_{\mathcal M_F (\mathbb R)} [0,\infty)$ defined by (\ref{Dysonsystem}) with parameter $\mass_n$ and branching structure given by the random trees $\hat\theta^n$ and initial condition
\begin{equation}
U_\emptyset(0)=0, \quad \text{where } \emptyset \text{ denotes the root of } \hat \theta^n.
\end{equation}
 As in \S\ref{sec:dyson}, denote $F(\hat \xi_t^n)=\exp{\langle \log \psi, \hat \xi_t^n \rangle}$.
 
 The term of the generator corresponding to the spatial motion will be identical to the unconditioned case considered in \S\ref{sec:dyson}, so for $\psi \in C_b^+(\hat{\mathbb R})\cap \mathcal D(\Delta)$, define
 \begin{equation}\label{spatial_generator}
 \mathcal J^n F\left(\xi_t\right)=F(\xi_t) \left(\mass_n H^0_{\xi_t}(\log \psi) + \frac{\mass_n}{\beta} \left \langle \frac{\psi''}{\psi}, \xi_t \right \rangle \right),
 \end{equation}
 which corresponds to the first two terms in (\ref{generator}). The generator for the branching will be given by (\ref{extra_term}).

\begin{theorem}\label{mart_conditioned}
The measure $\hat \xi^n_t$ solves the $( \mathcal G^n + \mathcal J^n , \hat \xi^n_t)$ martingale problem for
$\mathcal G^n$ and $\mathcal J^n$ defined by (\ref{extra_term}) and (\ref{spatial_generator}), respectively.
\end{theorem}

The bulk of the proof is supplied by the following lemma, which is the analog of Lemma \ref{mart_unconditioned_lemma}.

\begin{lemma}\label{mart_conditioned_lemma}
	\begin{equation}
	\frac{d}{dr} \mathbb E^n_{\delta_0} [F(\hat \xi_r^n)]\Big\vert_{r=t}=\mathbb E^n_{\delta_0} [\mathcal J^n F(\hat \xi_t^n)+\mathcal G_t^n F(\hat \xi_t^n)].
	\end{equation}
\end{lemma}

\begin{proof}[Proof of Lemma \ref{mart_conditioned_lemma}]
The varying offspring distribution gives a layer of complexity not present in the proof of Lemma \ref{mart_unconditioned_lemma}, so we present the argument in detail.
	
	We decompose based on the following events:
	\begin{equation}
	\begin{aligned}
	A^c &= \{ \text{no deaths in }[t, t+\delta t] \}\\
	A_0 &= \{ \text{one death in }[t, t+\delta t]\text{ and $0$ offspring} \}\\
	A_2 &= \{ \text{one death in }[t, t+\delta t]\text{ and $2$ offspring} \}.
	\end{aligned}
	\end{equation}
	The event that there is more than one death in $[t, t+\delta t]$ is $o(\delta t)^2$, so
	\begin{equation}\label{decomp}
	\begin{aligned}
	\frac{d}{dt} \mathbb E^n_{\hat \xi^n_t}\left[F(\hat \xi^n_t) \right]	
	=&  \lim_{\delta t\to 0} \mathbb P_{\delta_0}^n [A^c] 
	\frac{\mathbb E_{\delta_0}^n \left [ F(\hat \xi_{t+\delta t}) \vert A^c\right] - F(\hat \xi^n_t) }{\delta t}\\ 
	& + \lim_{\delta t\to 0} \mathbb P_{\delta_0}^n [A_0] 
	\frac{\mathbb E_{\delta_0}^n \left [ F(\hat \xi_{t+\delta t}) \vert A_0\right] - F(\hat \xi^n_t) }{\delta t}\\ 
	& + \lim_{\delta t\to 0} \mathbb P_{\delta_0}^n [A_2] 
	\frac{\mathbb E_{\delta_0}^n \left [ F(\hat \xi_{t+\delta t}) \vert A_2\right] - F(\hat \xi^n_t) }{\delta t} + o(\delta t)^2.
	\end{aligned}
	\end{equation}
	The $A^c$ term contributes the spatial generator $\mathcal J^k$.
	For the other two terms, we compute
	\begin{equation}
	\begin{aligned}
	\mathbb P_{\delta_0}^n[A_0]
	& =\mathbb P_{\delta_0}^n[A_0 \cup A_2] \,\,\mathbb P_{\delta_0}^n[A_0\vert A_0\cup A_2]\\
	&= \left(\mean_n \,N_t\, \delta t + o (\delta t) \right) \left(1-\Qk{t} \right),
	\end{aligned}
	\end{equation}
	and 
	\begin{equation}
	\mathbb P_{\delta_0}^n[A_2]=\left(\mean_n\, N_t\, \delta t + o (\delta t) \right) \Qk{t}.
	\end{equation}
	Substituting into (\ref{decomp}), the $A_0$ term is
	\begin{equation}\label{term2}
	\lim_{\delta t\to 0} \mathbb P_{\delta_0}^n [A_0] 
	\frac{\mathbb E_{\delta_0}^n \left [ F(\hat \xi_{t+\delta t}) \vert A_0\right] - F(\hat \xi^n_t) }{\delta t}
	= \mathbb E^n_{\delta_0} \left[ F(\hat \xi^n_t) \left \langle  \mean_n \left(1-\Qk{t} \right)(\frac{1}{\psi}-1), \hat\xi^n_t\right \rangle \right],
	\end{equation}
	and the $A_2$ term of (\ref{decomp}) is
	\begin{equation}\label{term3}
	\lim_{\delta t\to 0} \mathbb P_{\delta_0}^n [A_2] 
	\frac{\mathbb E_{\delta_0}^n \left [ F(\hat \xi_{t+\delta t}) \vert A_2\right] - F(\hat \xi^n_t) }{\delta t}
	= \mathbb E^n_{\delta_0} \left[ F(\hat \xi^n_t) \left \langle  \mean_n \,\Qk{t}\,(\psi-1), \hat \xi^n_t\right \rangle \right].
	\end{equation}. 
\end{proof}

\begin{proof}[Proof of Theorem \ref{mart_conditioned}]
	Again following [\cite{Etheridge}, \S1.2], we integrate the result of Lemma \ref{mart_conditioned_lemma}, 
	\begin{equation}
	\E^{n}_{\hat \xi^n_0} \left[F(\hat \xi^n_{t+u})-F(\hat \xi^n_t)\right]
	= 	\E^{n}_{\hat \xi^n_0} \left[\int_t^{t+u} \left(\mathcal J^n + \mathcal G^n \right)F(\hat \xi^n_{s})\,ds\right],
	\end{equation}
	where $\mathcal F_t$ denotes the natural filtration.
	
	Since starting at $\hat \xi^n_0$ and taking the conditional expectation given $\mathcal F_t$ and total population $n$ is equivalent to starting at $\hat \xi^n_t$ and taking the conditional expectation given total population $R^n_t$,
	\begin{equation}
	\begin{aligned}
	\E^n_{\hat \xi ^n_0} \left [ F(\hat \xi^n_{t+u}) - F(\hat \xi_t^n) \,  \vert \, \mathcal F_t \right]
	&=
	\E^{R^n_t}_{\hat \xi^n_t} \left [ F(\hat \xi^n_{u}) - F(\hat \xi_0^n) \right]\\
	&= \E^{R^n_t}_{\hat \xi^n_t} \left [ \int_0^{u} \left( \mathcal J^n + \mathcal G^n\right) F(\hat \xi^n_{s})\,ds  \right]\\
	& = \E^{n}_{\hat \xi^n_0} \left [ \int_t^{t+u} \left( \mathcal J^n + \mathcal G^n\right) F(\hat \xi^n_{s})\,ds  \vert \mathcal F_t\right],
	\end{aligned}
	\end{equation}
	which proves that
	\begin{equation} \label{eq:mart2}
	F(\hat \xi^n_t)-F(\hat \xi^n_0) - \int_0^t (\mathcal J^n + \mathcal G^n) F(\hat \xi^n_s)\, ds
	\end{equation}
	is a mean zero $\mathbb P^n_{\hat \xi^n_0}$ martingale for all $F$ of the form (\ref{eq:gen4}).
\end{proof}

To facilitate comparison with the results of \S\ref{sec:dyson}, denote
\be\label{eq:coshQ}
\coshQ (\phi) = \coshQ_{n,t} (\phi)= (1-\Qk{t}) e^\phi + \Qk{t} e^{-\phi},
\ee
and
\be
\Pn_t=1-2\Qk{t}.
\ee

\begin{lemma}\label{semimartingale_conditioned}
	For each $\phi\in C_b^+(\mathbb R) \cap \mathcal D (\Delta)$, the process $\left( \left \langle \phi, \hat\xi_t^k\right \rangle\right)_{t \geq 0}$ is a semimartingale that is the sum of a predictable finite variation process
	\be \label{eq:var(cond)}
	\hat V_t(\phi) = \langle \phi, \hat \xi^n_0\rangle
		+ \mass_n\int_0^t  \left( H_{\hat \xi_s}(\phi)
		+\left( \frac{1}{\beta}-\frac{1}{2}\right)  \left \langle\phi'' ,\, \hat\xi^n_s\right\rangle
		-  \frac{\mean_n}{\mass_n}\, \Pn_s\,\left\langle\phi
		,\, \hat\xi^n_s\right\rangle \right) ds,
	\ee
and a martingale, $\hat M_t(\phi)$, with quadratic variation
	\begin{equation}\label{quadratic_var_k_conditioned}
	[\hat { M} (\phi)]_t = 2\int_0^t \left \langle  \frac{\mass_n}{\beta} \left(\phi' \right)^2 + \mean_n \left(\coshQ (\phi)-1\right) + \mean_n\,\Pn_s\, \phi , \,\hat \xi^n_s \right \rangle ds.
	\end{equation}
\end{lemma}
Note that $\coshQ$ depends on $s$; see (\ref{eq:coshQ}).
\begin{proof}
	
	The argument is similar to the proof of Lemma \ref{semimartingale_unconditioned} (the analogous result for the unconditioned case), with the difference that evaluating $\frac{d}{d\theta} \mathcal G^n F_\theta(\hat \xi^n_s)\big\vert_{\theta=0}$ gives an additional $\phi$ term.
	
	We let $\theta\geq 0$, $\psi = e^{-\theta \psi}$, and substitute the test function $F_\theta(\hat \xi^n_t)=\exp \langle -\theta \phi, \hat \xi^n_t \rangle$ into equation (\ref{eq:mart2}). Taking the expectation, differentiating with respect to $\theta$, and evalutating at $\theta=0$, we obtain
	\be \label{eq:mart(conditioned)}
	0=\E_{\hat \xi^n_0}^n \left [ \langle \phi, \hat \xi^n_{t+u}\rangle -\langle \phi, \hat \xi^n_t \rangle + \int_t^{t+u} \frac{d}{d\theta}\left[ \left( \mathcal J^n + \mathcal G^n   \right) F_\theta(\hat \xi^n_s) \right]\Big\vert_{\theta=0} \, ds \right].
	\ee
	We compute the integrand using Theorem \ref{mart_conditioned}
	\be\label{eq:gen(cond)}
	\begin{aligned}
	\left( \mathcal J^n + \mathcal G^n \right)F_\theta(\hat \xi^n_s) &= 
	\exp \langle -\theta \phi, \hat \xi^n_s \rangle \times\\
	&\left( -\theta \mass_n H^0_{\hat \xi^n_s} (\phi) + \left \langle \frac{\mass_n}{\beta} (\theta^2 (\phi ')^2) -\theta \phi '') + \mean_n (\coshQ \theta \phi -1) + \mean_n \, \Pn\, \phi, \hat \xi^n_s\right \rangle \right). 
	\end{aligned}
	\ee
	Now differentiate with respect to $\theta$ and use the definition of $H$ in equation (\ref{eq:gen2}) to obtain
	\be \label{eq:deriv(conditioned)}
	\frac{d}{d\theta} \left[\left( \mathcal J^n + \mathcal G^n   \right) F_\theta (\hat \xi^n_s) \right]\Big\vert_{\theta=0} 
	= \mass_n   \left( H_{\hat \xi_s}(\phi)
	+\left( \frac{1}{\beta}-\frac{1}{2}\right)  \left \langle\phi'' ,\, \hat\xi^n_s\right\rangle
	-  \frac{\mean_n}{\mass_n}\, \Pn_s\,\left\langle\phi
	,\, \hat\xi^n_s\right\rangle \right).
	\ee
	We substitute equation (\ref{eq:deriv(conditioned)}) into (\ref{eq:mart(conditioned)}) to conclude that $\langle \phi, \hat \xi^n_t\rangle$ is a semimartingale with predictable finite variation process $\hat V_t(\phi)$ given in (\ref{eq:var(cond)}).
	
	The quadratic variation of $\hat M_t(\phi)$ is computed as follows. Since $\langle \phi, \hat \xi^n_t\rangle$ is a semimartingale, we apply It\^o's formula backwards to conclude that
	\begin{equation}\label{mart_for_quadratic_var_conditioned}
	\exp(-\langle \phi, \hat\xi^n_t)-\exp(-\langle \phi,\hat \xi^n_0) 
	+ \int^t_0 \exp(-\langle \phi, \hat\xi_s\rangle) \, d\hat V_s
	-\frac{1}{2}\int^t_0 \exp(-\langle \phi, \hat\xi^n_s\rangle)\,d[\hat M(\phi)]_s
	\end{equation}
	is a martingale. 
	On the other hand, letting $\theta=1$ in (\ref{eq:gen(cond)}), and using (\ref{eq:var(cond)}), Theorem \ref{mart_conditioned} implies that 
\be
\begin{aligned}
\exp (-\langle \phi, \hat\xi^n_t\rangle)-\exp(\langle -\phi, \hat\xi^n_0\rangle) 
&+ \int_0^t \exp\langle -\phi, \hat \xi^n_s\rangle \, d\hat V_s\\
&- \int_0^t  \exp\langle -\phi, \hat \xi^n_s\rangle \left \langle \frac{\mass_n}{\beta} (\phi ')^2 + \mean_n(\coshQ \phi -1) + \mean_n\, \Pn_s \phi  , \hat \xi^n_s \right \rangle   \, ds
\end{aligned}
\ee
	is a martingale. 
	Comparing to (\ref{mart_for_quadratic_var_conditioned}), we obtain (\ref{quadratic_var_k_conditioned}).
\end{proof}

\subsection{Tightness}

 For each $n\geq 1$, let $\hat\theta^n$ be distributed as a critical binary Galton-Watson tree with independent exponential lifetimes of mean $\mean_n$ conditioned to have $n$ total individuals. 
We will need to restrict to a localization set similar to that used in \cite{Serlet}.  Let $\epsilon'$ be small and let $\sigma$ denote the stopping time
 \begin{equation}\label{stopping_time}
 \sigma=\inf_{t\geq 0}\left\{t: \frac{R^n_t-1}{n}\leq \epsilon' \right\}.
 \end{equation}
 Let $\hat \mu^n$ be the  measure-valued process satisfying (\ref{Dysonsystem}) with parameter $\mass_n$ stopped at time $\sigma$.

\begin{theorem}\label{tightness_conditioned}
If $\mass_n=\mean_n=\frac{1}{2\sqrt n}$, then the sequence of conditioned processes $\{\hat \mu^n\}_{n\geq 1}$ is tight in $D_{\mathcal M_F(\hat {\real})}[0,\infty)$.
\end{theorem}

\begin{proof}[Proof of Theorem \ref{tightness_conditioned}]
	The proof follows the same method as the proof of the unconditioned case (Theorem \ref{tightness_unconditioned}), though in this case the total mass proccess converges to the local time of the normalized Brownian excursion, rather than to the Feller diffusion, so the first step takes a bit more work.
	
	\textbf {1.}
	As already noted, $2 \langle 1, \hat \mu^n_t  \rangle= (n)^{-{1/2}} \hat N^n_t \distr \brexlocal^t$
The righthand side of (\ref{LocalTimeFormula}) is given by the well-known Kolmogorov-Smirnov formula:
	\begin{equation}
	P\left( \sup_{0\leq t \leq 1} B^{\abs{\text{br}}, 1}_t \leq x \right)
	=1+ 2 \sum_{i=-\infty}^\infty (-1)^i  e^{-2i^2x^2},
	\end{equation}
	implying that for each $\epsilon>0$ there exists a $K_\epsilon$ such that
	\begin{equation}
	\mathbb P \left( \sup_{t\geq 0} \langle 1, \hat \mu^k_t\rangle\geq K_{\epsilon}  \right) <\epsilon.
	\end{equation}
	
	\textbf {2.}
	As in \ref{tightness_conditioned}, we need to verifty that given a sequce of stopping times $\tau_n$, bounded by $T$, for each $\epsilon>0$ there exists $\delta>0$ and $n_0$ such that
	\begin{equation}\label{V_conditioned}
	\sup_{n\geq n_0} \sup_{t\in [0, \delta]} \mathbb P\left [ \abs{\hat V^{(n)} (\tau_n + t) -\hat V^{(n)} (\tau_n ) } >\epsilon  \right] \leq \epsilon,
	\end{equation}
	and 
	\begin{equation}\label{M_conditioned}
	\sup_{n\geq n_0} \sup_{t \in [0, \delta]} \mathbb P\left [ \abs{ [\hat M^{(n)}]_{\tau_n + t} - [\hat M^{(n)}]_{\tau_n} } >\epsilon  \right] \leq \epsilon,
	\end{equation}
	where 
	\begin{equation}
	\langle \phi, \hat \mu^n_s\rangle = \hat M^n_s + \hat V^n_s,
	\end{equation} as in Lemma \ref{semimartingale_conditioned}.
	
	Since $\beta \geq 1$, we compute
	\begin{equation}
	\begin{aligned}
	\mathbb E \left[ \int_{\tau_n}^{\tau_n+\theta} \abs{ \left \langle \frac{1}{\beta}\phi''(x), \hat \mu^n_s(x)\right \rangle } ds\right]
	& \leq \norm{\phi''}_\infty \mathbb E \left[ \int_{\tau_n}^{\tau_n+\theta}\langle 1, \hat\mu^n_s \rangle  ds \right]\\
	& < \norm{\phi''}_\infty \theta \mathbb E \left[ \frac{\max_{s\geq 0} \hat N_{s}^n}{\sqrt n}   \right]\\
	& < \norm{\phi''}_\infty \frac{\theta}{2} \left(\mathbb E \left[ \sup_{s\geq 0} L^s_{\mathbbm e} \right] + c_0 \right),
	\end{aligned}
	\end{equation}
	for a uniform constant $c_0$. 
	Since we assume that $\norm{\phi''}_\infty<\infty$, the $\phi'$ and $\phi''$ terms may be bounded using the same argument.
	
	We only consider the processes up to time $\sigma$, but since the estimates hold for arbitrary $\epsilon'$, this is sufficient to prove tightness.
	
	We may expand
	\begin{equation}
	\begin{aligned}
	\hat N^n_s(1-2\Qk{s}) &=\frac{(\hat N^n_s)^2-R^n_s}{R^n_s-1}\\
	&=  \frac{(\hat N^n_s)^2/n}{(R^n_s -1)/n}-\frac{R^n_s}{R^n_s-1}
	\end{aligned}
	\end{equation}
	where each term is in $(0, 1)$,
	so that
	\begin{equation}\label{fin_var_term_estimate}
	\begin{aligned}
	\mathbb E \left[ \int_{\tau_n}^{\tau_n+\theta} \abs{ \left \langle \sqrt n (1-2 \Qk{s})\, \phi, \hat \mu^n_s(x)\right \rangle } ds\right]
	& \leq \norm {\phi}_\infty \mathbb E \left[ \int_{\tau_n}^{\tau_n+\theta} \abs{ \hat N^n_s (1-2\Qk{s}) } ds\right]\\
	& \leq \norm {\phi}_\infty \mathbb E \left[ \int_{\tau_n}^{\tau_n+\theta} \abs{ \frac{(\hat N^n_s)^2/n}{(R^n_s -1)/n}-\frac{R^n_s}{R^n_s-1} } ds\right]\\
	& \leq \norm {\phi}_\infty \mathbb E \left[ \int_{\tau_n}^{\tau_n+\theta} \left(\frac{(\hat N^n_s)^2}{n} \frac{1}{\epsilon'} + 2 \right)ds\right].
	\end{aligned}
	\end{equation}
	Again appealing to an estimate on the local time, it is possible to choose $\theta$ sufficiently small and $n$ sufficiently large (depending on $\epsilon'$) so that (\ref{fin_var_term_estimate}) is less than $\epsilon$, completing the proof.
	
	\end{proof}

\subsection{The martingale problem for the conditioned Dyson superprocess}

\begin{proof} [Proof of Theorem \ref{LimitingMartingaleProblem_Conditioned}]
	As in the proof of Theorem \ref{LimitingMartingaleProblem}, we take the limit along a sequence of test functions and then show that this limit holds in general. Letting $\psi_n = 1- \frac{\phi}{2\sqrt n}$ in Lemma  \ref{semimartingale_conditioned} and substituting $2\sqrt n \,\hat \mu^n_t= \hat \xi^n_t$ shows that for each $n$,

	\begin{equation}\label{mart_conditioned_subsequence}
	\begin{aligned}
	&\left \langle 2 \sqrt n \log {\left( 1-\frac{\phi}{2\sqrt n}\right)}, \hat\mu_t^n \right \rangle 
	-\left \langle 2\sqrt n \log \left( 1-\frac{\phi}{2\sqrt n}\right), \hat \mu_0^n \right \rangle \\
	& \qquad \qquad \qquad \qquad-\int_0^t \left \langle  -\frac{1}{2} \int_{\mathbb R} \frac{\frac{\phi'(x)}{1-\phi(x)/2\sqrt n} -\frac{\phi'(y)}{1-\phi(y)/2\sqrt n}}{x-y} \hat\mu^n_s(dy) 
	, \,\hat \mu^n_s(x) \right \rangle \, ds\\
	&\qquad \qquad\qquad \qquad -\int_0^t \left \langle  \frac{-\frac{\mass_n}{\beta} \left( \phi'' - \frac{\phi'' \phi }{2\sqrt n} - \frac{(\phi')^2}{2\sqrt n}\right)}{\left( 1-\frac{\phi}{2\sqrt n}\right)^2}
	-4 n \log\left(1-\frac{\phi}{2\sqrt n}\right) \left( 1-2\Qk{s}\right)
	, \, \hat\mu^n_s(x) \right \rangle \, ds
	\end{aligned}
	\end{equation}
	is a martingale with quadratic variation
	\begin{equation}\label{quad_var_conditioned_subsequence}
	\int_0^t \left \langle  \frac{\mass_n}{\beta\sqrt n} \frac{ (\phi')^2}{\left( 1-\frac{\phi}{2\sqrt n} \right)^2}
	+ \frac{ 4\sqrt n \left (2\Qk{s}-1\right) \phi}{1-\frac{\phi}{2\sqrt n}} + \frac{2(1-\Qk{s}) \phi^2}{1-\frac{\phi}{2\sqrt n}} + 8 n (2\Qk{s}-1) \log{\left(1-\frac{\phi}{2\sqrt n}\right)}
	, \, \hat \mu^n_s \right \rangle ds.
	\end{equation}
	To compute the limit of (\ref{mart_conditioned_subsequence}), expand
	\begin{equation}
	\begin{aligned}
	2\sqrt n \left(1-2\Qk{s}\right)
	&=2\sqrt n \frac{(\hat N^n_s)^2 -R^n_s}{\hat N^n_s(R^n_s-1)}\\
	&= 2\sqrt n \frac{(\hat N^n_s)^2 -R^n_s}{\hat N^n_s(R^n_s)} \left( 1 + \frac{1}{R^n_s-1}\right).
	\end{aligned}
	\end{equation}
	As we only consider the process up to time $\sigma$, 
	\begin{equation}
	\frac{1}{R^n_s -1} < \frac{1}{n \epsilon'}.
	\end{equation}
	We have
	\begin{equation}
	2\sqrt n \frac{(\hat N^n_s)^2 -R^n_s}{\hat N^n_s(R^n_s)} =\frac{2\sqrt n\hat N^n_s}{R^n_s} - \frac{2\sqrt n}{\hat N^n_s}.
	\end{equation}
	The second term above has limit 
	\begin{equation}
	\lim_{n\to \infty }\frac{2\sqrt n }{\hat N^n_t} = \frac{4}{L^s_{\mathbbm e}}.
	\end{equation}
	To find the limit of the first term, define $D^n_t$ by
	\begin{equation}
	D^n_t = k -R^n_t,
	\end{equation}
	so that $D^n_t$ records the number of individuals who have already died by time $t$. Then
	\begin{equation}
	 \frac{2\sqrt n \hat N^n_s}{R^n_s}= \frac{2\hat N^n_s/\sqrt n}{ 1-D^n_t/n}.
	\end{equation}
	However,
		\begin{equation}
	\frac{D^n_s}{n} \distr \int_0^s L^r_{\mathbbm e}dr.
	\end{equation}
	which can be seen by comparing $D^n_t/2\sqrt n$ to $\int_0^t\hat N^n_s\, ds$ and noticing that $D^n_s$ counts each individual once, rather than weighting it according to its lifetime, which has mean $\frac{1}{2\sqrt n}$. 
	
	Therefore, the limit of the martingale  (\ref{mart_conditioned_subsequence}) is 
	\begin{equation}
	\hat M_t=  \langle \phi, \hat \mu_t\rangle -\langle \phi, \hat \mu_0\rangle - \int_0^t\left( \int_{\mathbb R} \int_{\mathbb R} \frac{\phi'(x)-\phi'(y)}{x-y} \hat \mu_s(dx) \hat \mu_s(dy) + \left \langle  \phi\left(\frac{4}{L^s_{\mathbbm e}} - \frac{L_{\mathbbm e}^s}{1-\int_0^s L_{\mathbbm e}^r \, dr}\right), \, \hat \mu_s  \right \rangle \right) ds.
	\end{equation}
	To take the limit of the quadratic variation (\ref{quad_var_conditioned_subsequence}), notice that the second and fourth terms cancel since
	\begin{equation}
	\lim_{n\to \infty}\sqrt n \log \left(1-\frac{\phi}{\sqrt n}\right) =-\phi.
	\end{equation}
	 Furthermore, 
	 the computation above of $\lim_{n\to \infty} \sqrt n (1-2\Qk{s})$ implies that $2\Qk{s}\to 1$ pointwise, so rewriting 
	 \begin{equation}
	 2(1-\Qk{s})=1+(1-2\Qk{s}),
	 \end{equation}
	  and using the dominated convergence theorem, we conclude that the limit of (\ref{quad_var_conditioned_subsequence}) is
	\begin{equation}
	\int_0^t \left \langle  \phi^2, \hat \mu_s\right \rangle ds.
	\end{equation}
	That the martingale property persists in the limit and that the limit lies in $\mathcal M_F(\mathbb R)$ are analogous to the corresponding steps in the proof of Theorem \ref{LimitingMartingaleProblem}.
\end{proof}

\bibliography{hm1}
\bibliographystyle{siam}

\bibliography{hm1}
\bibliographystyle{siam}

\end{document}